\documentclass[a4paper,reqno]{amsart}
\usepackage[utf8]{inputenc}
\usepackage[T1]{fontenc}
\usepackage{lmodern}
\usepackage{amssymb,amsxtra}
\usepackage[all]{xy}
\usepackage{nicefrac,mathtools,enumitem}
\usepackage{mdwlist}
\setlist[enumerate,1]{label=\textup{(\arabic*)}}
\usepackage{lmodern,textcomp}

\usepackage{tikz}
\usetikzlibrary{matrix}
\tikzset{cd/.style=matrix of math nodes,row sep=2em,column sep=2em, text height=1.5ex, text depth=0.5ex}
\tikzset{cdar/.style=->,auto}
\tikzset{overar/.style={draw=white,double=black,double distance=.4pt,very thick}}

\usetikzlibrary{positioning,calc}
\usetikzlibrary{shapes}
\usetikzlibrary{matrix}
\usetikzlibrary{arrows}
\usepackage{microtype}
\usepackage[pdftitle={Braided quantum },
  pdfauthor={Sutanu Roy},
  pdfsubject={Mathematics}
]{hyperref}

\usepackage[lite]{amsrefs}

\renewcommand{\PrintDOI}[1]{\href{http://dx.doi.org/\detokenize{#1}}{doi: \detokenize{#1}}}

\BibSpec{book}{%
  +{}  {\PrintPrimary}                {transition}
  +{,} { \textit}                     {title}
  +{.} { }                            {part}
  +{:} { \textit}                     {subtitle}
  +{,} { \PrintEdition}               {edition}
  +{}  { \PrintEditorsB}              {editor}
  +{,} { \PrintTranslatorsC}          {translator}
  +{,} { \PrintContributions}         {contribution}
  +{,} { }                            {series}
  +{,} { \voltext}                    {volume}
  +{,} { }                            {publisher}
  +{,} { }                            {organization}
  +{,} { }                            {address}
  +{,} { \PrintDateB}                 {date}
  +{,} { }                            {status}
  +{}  { \parenthesize}               {language}
  +{}  { \PrintTranslation}           {translation}
  +{;} { \PrintReprint}               {reprint}
  +{.} { }                            {note}
  +{.} {}                             {transition}
  +{} { \PrintDOI}                   {doi}
  +{} { available at \url}            {eprint}
  +{}  {\SentenceSpace \PrintReviews} {review}
}

\numberwithin{equation}{section}

\theoremstyle{plain}
\newtheorem{theorem}[equation]{Theorem}

\newtheorem{proposition}[equation]{Proposition}

\newtheorem{corollary}[equation]{Corollary}

\theoremstyle{definition}
\newtheorem{definition}[equation]{Definition}

\theoremstyle{remark}
\newtheorem{remark}[equation]{Remark}

\newtheorem{example}[equation]{Example}

\newcommand*{\Braiding}[2]{\begin{tikzpicture}[baseline]
    \draw[-] (0,0) -- (1.4ex,1.4ex) node[right,inner sep=0pt] {$\scriptstyle #2$};
    \draw[-,draw=white,line width=2.4pt] (0,1.4ex) -- (1.4ex,0);
    \draw[-] (1.4ex,0) -- (0,1.4ex) node[left,inner sep=0pt] {$\scriptstyle #1$};
  \end{tikzpicture}}
\newcommand*{\Dualbraiding}[2]{\begin{tikzpicture}[baseline]
    \draw[-] (1.4ex,0) -- (0,1.4ex) node[left,inner sep=0pt] {$\scriptstyle #1$};
    \draw[-,draw=white,line width=2.4pt] (0,0) -- (1.4ex,1.4ex);
    \draw[-] (0,0) -- (1.4ex,1.4ex) node[right,inner sep=0pt] {$\scriptstyle #2$};
  \end{tikzpicture}}

\newcommand*{\Rmattxt}{R}%R-matrix for text use

\newcommand*{\Bialg}[1]{(#1,\Comult[#1])}%C*-bialgebra
\usepackage{dsfont}

\newcommand*{\nb}{\nobreakdash}
\newcommand*{\Star}{$^*$\nb-}

\newcommand*{\C}{\mathbb C}
\newcommand*{\Z}{\mathbb Z}

\newcommand*{\N}{\mathbb N}

\newcommand*{\T}{\mathbb T}

\newcommand*{\G}[1][G]{\mathbb #1}%quantum group

\newcommand*{\Comult}[1][]{\Delta_{#1}}%comultiplication
\newcommand*{\Qgrp}[2]{\mathbb{#1}=(#2,\Comult[#2])}%quantum group
\newcommand*{\Bound}{\mathbb B}%adjointable operators on a Hilbert module
\newcommand*{\Mat}[1]{\textup{M}_{#1}}%Matrix algebra
\newcommand*{\Contvin}{\textup C_0}%continuous functions vanishing at infinity
\newcommand*{\Cont}{\textup C}%continuous functions

\newcommand*{\Mor}{\textup{Mor}}%nondegenerate *-homomorphisms of C*-algebras
\newcommand*{\Id}{\textup{id}}%identity map

\newcommand*{\bichar}{\textup V}%bicharacter viewed as an element of of the unitary multiplier
\newcommand*{\Rmat}{\textup R}%R-matrix
\newcommand*{\Cst}{\textup C^*}%C*-algebra
\newcommand*{\Cred}{\textup C^*_\textup r}%reduced group C*-algebra

\newcommand*{\Cstcat}{\mathfrak{C^*alg}}%category of C*-algebras
\newcommand*{\Hils}[1][H]{\mathcal{#1}}%Hilbert space
\newcommand*{\U}{\mathcal U}%unitary group

\newcommand*{\defeq}{\mathrel{\vcentcolon=}}

\newcommand*{\norm}[1]{\lVert#1\rVert}

\DeclareMathOperator{\Aut}{Aut}% automorphism group of a C*-algebra

\allowdisplaybreaks

\begin{document}
\title[Homogeneous quantum symmetries of finite quantum spaces over~$\mathbb{T}$]{Homogeneous quantum symmetries of finite spaces over the circle group}

\author{Sutanu Roy}
\email{sutanu@niser.ac.in}
\address{School of Mathematical Sciences\\
 National Institute of Science Education and Research  Bhubaneswar, HBNI\\
 Jatni, 752050\\
 India}

\begin{abstract}
Suppose \(D\) is a finite dimensional \(\Cst\)\nb-algebra carrying a continous action~\(\overline{\Pi}\) of the circle group~\(\T\). We study the quantum symmetry group of~\(D\), taking~\(\overline{\Pi}\) into account.  We show that they are braided compact quantum groups~\(\G\) over~\(\T\). Here, the \(\Rmattxt\)\nb-matrix, \(\Z\times\Z\ni(m,n)\to 
\zeta^{-m\cdot n}\in\T\), for a fixed~\(\zeta\in \T\), governs the braided structure. In particular, if~\(\overline{\Pi}\) is trivial,~\(\zeta=1\) or~\(D\) is commutative, then \(\G\) coincides with Wang's quantum group of automorphisms of~\(D\). Moreover, we show that the bosonisation of \(\G\) corresponds to the quantum symmetry group of the crossed product \(\Cst\)\nb-algebra~\(D\rtimes\Z\), where the~\(\Z\)\nb-action is generated by~\(\overline{\Pi}_{\zeta{^{-1}}}\).
\end{abstract}

\thanks{The author was partially supported by INSPIRE faculty award given by D.S.T., Government of India grant no. DST/INSPIRE/04/2016/000215. The author is grateful to  Suvrajit Bhattacharjee, Soumalya Joardar and Adam Skalski for stimulating discussions.}
\maketitle

\section{Introduction}
 \label{sec:Intro}
 
In the realm of compact quantum groups of Woronowicz~\cites{W1987a,W1998a}, the inception of the theory of quantum symmetry groups of classical and quantum space goes back to the seminal paper of Wang~\cite{W1998}. In that article, Wang studied quantum symmetry groups of finite sets and finite dimensional noncommutative \(\Cst\)\nb-algebras or finite quantum spaces equipped with a reference state. They are extended to several discrete structures by Banica, Bichon and several others (\cites{B2005,B2005a,B2003,JM2018,SW2018} and the references therein). On the other hand, Goswami~\cite{G2009} introduced the continuous counterpart, namely, the quantum isometry group for spectral triples. He and his collaborators studied it extensively~\cites{BG2009a,BG2010a,GJ2018,G2020}. Goswami's work motivated Banica and Skalski to develop a general framework of quantum symmetry groups based on orthogonal filtrations of unital \(\Cst\)\nb-algebras~\cite{BS2013}. 

Semidirect product construction of groups is a fundamental way of extending some homogeneous symmetries to inhomogeneous symmetries of a physical system. This perspective suggests that, within the framework of the noncommutative geometry, semidirect product construction for quantum groups might be a natural way to obtain inhomogeneous symmetries of quantum spaces or \(\Cst\)\nb-algebras from their homogeneous symmetries. Several interesting examples of unital \(\Cst\)\nb-algebras~\(D\) naturally carry a continuous action~\(\gamma\) of the circle group~\(\T\), that preserves a distinguished faithful state, say~\(\phi\), on \(D\). Fix~\(\zeta\in\T\). We view the circle group~\(\T\) a quasitriangular quantum group with respect to the \(\Rmattxt\)\nb-matrix~\(\Rmat\), which is a bicharacter on~\(\Z\), defined by 
\begin{equation}
 \label{eq:Rmat}
 \Rmat(l,m)=\zeta^{-lm}, 
 \qquad\text{for all~\(l,m\in\Z\).}
\end{equation}
The quantum symmetries of compact type of \(D\) encoding and respecting the information of the action~\(\gamma\) are the \(\T\)\nb-\emph{homogeneous compact quantum symmetries} of the system \((D,\gamma,\phi)\). If such an object exists, it becomes a \emph{braided compact quantum group} \(\G\) over~\(\T\) (see~\cite{MRW2016}*{Definition 6.1}). The \(\T\)\nb-action~\(\gamma\) and the above \(\Rmattxt\)\nb-matrix govern the braided structure of~\(\G\). Moreover, we expect to capture the inhomogeneous compact quantum symmetries of \((D,\gamma,\phi)\) in terms of the bosonisation of \(\G\) (see~\cite{MRW2016}*{Theorem 6.4}). 

This article aims to support the above model for finite dimensional \(\Cst\)\nb-algebras \(D\) carrying an action of~\(\T\). If~\(D\) is commutative with~\(\textup{dim}(D)=n\) for some natural number~\(n\), then it is generated by \(n\) many commuting projections~\(p_{1},\cdots p_{n}\) satisfying~\(\sum_{i=1}^{n}p_{i}=1\). If~\(\gamma\) is  a continuous action of~\(\T\) on~\(D\) then~\(\gamma_{z}(p_{i})\) is self adjoint and \(\gamma_{z}(p_{i})=\gamma_{z}(p_{i})^{2}\) for all~\(z\in\T\) and~\(1\leq i\leq n\). This forces~\(\gamma\) to become trivial. Therefore, the quantum symmetry group of~\(D\) corresponds to Wang's quantum permutation group~\cite{W1998}*{Theorem 3.1}. So, we consider~\(D\) to be noncommutative, that is, \(D\) is isomorphic to a finite direct sum of matrix algebras.
 
In particular, let~\(D=\Mat{n}(\C)\) for a fixed number~\(n\).
The matrix algebra~\(\Mat{n}(\C)\) is the universal \(\Cst\)\nb-algebra generated by 
 \(\{E_{ij}\}_{1\leq i,j\leq n}\) subject to the following conditions: 
\begin{equation}
 \label{eq:gen_Mat}
  E_{ij}E_{kl}=\delta_{j,k}\cdot E_{il},
  \qquad 
  E_{ij}^{*}=E_{ji},
  \qquad 
 \sum_{i=1}^{n}E_{ii}=1.
\end{equation}
Let~\(d_{1}\leq d_{2}\leq \cdots \leq d_{n}\) be integers.
Then it is easy to verify that~\(\Pi\), defined by 
\begin{equation}
 \label{eq:T-act-Mat}
  \Pi_{z}(E_{ij})\defeq z^{d_{i}-d_{j}}E_{ij} \quad\text{for all \(z\in\T\),}
\end{equation}
is a continuous action of~\(\T\) on~\(\Mat{n}(\C)\).   
Let~\(\phi\) be the canonical normalised trace on~\(\Mat{n}(\C)\), that is, \(\phi(E_{ij})=\frac{1}{n}\cdot \delta_{i,j}\) for all~\(1\leq i,j\leq n\). The action \(\Pi\) of~\(\T\) on~\(\Mat{n}(\C)\) preserves \(\phi\), that is, \(\phi\circ \Pi_{z}(\cdot)=\phi(\cdot) 1\). 

Consider the braided tensor product~\(\boxtimes_{\zeta}\) of unital \(\Cst\)\nb-algebras 
with~\(\T\)\nb-action discussed in Section~\ref{subsec:braidedtensor}. The \emph{braided compact quantum symmetries} of the triple \((\Mat{n}(\C),\Pi,\phi)\) are  \(\phi\) preserving actions \(\eta\colon\Mat{n}(\C)\to\Mat{n}(\C)\boxtimes_{\zeta}B\) of braided compact quantum groups~\(\G=(B,\beta,\Comult[B])\) over~\(\T\) on~\((\Mat{n}(\C),\Pi)\) in the sense of the Definitions~\ref{def:br_CQG} and \ref{def:br_act}, respectively. In fact, \(\eta\) is uniquely determined by its matrix coefficients~\(\{b^{ij}_{kl}\}_{1\leq i,j,k,l\leq n}\subseteq B\):
 \begin{equation}
  \label{eq:Mn-action}
  \eta(E_{ij})=\sum_{k,l=1}^{n}E_{kl}\boxtimes_{\zeta} b^{kl}_{ij}, 
  \qquad\text{for all~\(1\leq i,j\leq n\).}
 \end{equation}
We construct a braided compact quantum group, denoted by~\(\Aut(\Mat{n}(\C),\Pi,\phi)\), in Theorem~\ref{the:qnt_aut_matrx}. Subsequently, we observe in Proposition~\ref{prop:action_qnt_on_Matx} that it acts faithfully on~\((\Mat{n}(\C),\Pi)\) and preserves~\(\phi\). 

According to~\cite{MRW2016}*{Section 6}, every braided compact quantum group over~\(\T\) uniquely corresponds to an ordinary compact quantum group. In the purely algebraic setting, this was discovered by Radford~\cite{R1985} and extensively studied by Majid~\cite{M1994} under the name \emph{bosonisation}. In Section~\ref{sec:Boson}, we construct the bosonisation of \(\Aut(\Mat{n}(\C),\Pi,\phi)\), denoted by~\(\textup{Bos}(\Aut(\Mat{n}(\C),\Pi,\phi))\).

In Proposition~\ref{prop:bos_act_lift}, we observe that \(\textup{Bos}(\Aut(\Mat{n}(\C),\Pi,\phi))\) acts on~\(\Mat{n}(\C)\boxtimes_{\zeta}\Cont(\T)\), extending the (braided) action of~\(\Aut(\Mat{n}(\C),\Pi,\phi)\) on \((\Mat{n}(\C),\Pi)\) and the action of~\(\T\) on~\(\Cont(\T)\) by translation. Then we use it to prove the universal property of~\(\Aut(\Mat{n}(\C),\Pi,\phi)\) in Theorem~\ref{the:univ_obj}. 

The \(\Rmattxt\)\nb-matrix~\eqref{eq:Rmat} induces an action of~\(\Z\) on~\(\Mat{n}(\C)\) generated by~\(\Pi_{\zeta^{-1}}\). Then the associated crossed product~\(\Mat{n}(\C)\rtimes\Z\) becomes isomorphic to \(\Cont(\T)\boxtimes_{\zeta}\Mat{n}(\C)\). Consider the orthogonal filtrations~\(\widetilde{\Mat{n}(\C)}\) and~\(\widetilde{\Cont(\T)}\) of~\(\Mat{n}(\C)\) and \(\Cont(\T)\) given in Example~\ref{ex:Mat_filt} and Example~\ref{ex:Circle_filt}, respectively. Then we apply \cite{BMRS2018}*{Proposition 4.9} to construct an orthogonal filtration of~\(\Cont(\T)\boxtimes_{\zeta}\Mat{n}(\C)\). Consequently, in Proposition~\ref{prop:Boson_isometry}, we prove that the corresponding quantum symmetry group of~\(\Cont(\T)\boxtimes_{\zeta}\Mat{n}(\C)\) is isomorphic to~\(\textup{Bos}(\Aut(\Mat{n}(\C),\Pi,\phi))\). This is the universal property of the bosonisation of~\(\Aut(\Mat{n}(\C),\Pi,\phi)\). Furthermore, combining it with~\cite{BMRS2018}*{Theorem 5.1} we also observe~\(\textup{Bos}(\Aut(\Mat{n}(\C),\Pi,\phi))\cong\mathfrak{D}_{\bichar}\) in Corollary~\ref{cor:bos_dinf}, where~\(\mathfrak{D}_{\bichar}\) is the Drinfeld's double of the quantum symmetry groups \(\textup{QISO}(\widetilde{\Cont(\T)})\) and~\(\textup{QISO}(\widetilde{\Mat{n}(\C)})\) with respect to some bicharacter~\(\bichar\) suitably chosen. 

Within the scope of this article, the bosonisation picture has significant advantages over the Drinfeld's double, both computationally and conceptually. Indeed, the concrete computation of the bosonisation of a braided compact quantum group over~\(\T\) is more straightforward than the computation of Drinfeld’s double. Secondly, in Remark~\ref{rem:inhomogeneous}, we observe that~\(\Aut(\Mat{n}(\C),\Pi,\phi)\) is the homogeneous quantum symmetry group of the \(\Cst\)\nb-dynamical system \((\Mat{n}(\C),\Z,\Pi_{\zeta^{-1}})\) obtained from \((\Mat{n}(\C),\Pi)\). At the same time, its bosonisation (semidirect product) corresponds to the inhomogeneous (extending \(\T\)\nb-homogeneous) quantum symmetry group of the same system. Conceptually, this approach is more natural while computing the quantum symmetry of the given the \(\Cst\)\nb-dynamical system~\((\Mat{n}(\C),\Z,\Pi_{\zeta^{-1}})\).

Finally, in Section~\ref{sec:direct_sum}, we employ these analyses to study similar results for the direct sum of matrix algebras. As a byproduct, we observe in Remark~\ref{rem:comm_finpt} that a finite classical space, that is,~\(D\cong\C^{m}\) for any natural number~\(m\), has no braided quantum symmetries (over~\(\T\)), as expected.

\section{Preliminaries}
All Hilbert spaces and \(\Cst\)\nb-algebras (which are not explicitly multiplier algebras) are assumed to be separable. For any subsets~\(X,Y\) of a given unital~\(\Cst\)\nb-algebra~\(D\), the norm closure of \(\{xy\mid x\in X, y\in Y\}\subseteq D\) is denoted by~\(XY\). We use \(\otimes\) for both the tensor product of Hilbert spaces and the minimal tensor product of \(\Cst\)\nb-algebras, which is well understood from the context.
\label{prelim}
 \subsection{Compact quantum groups and their actions}
 A \emph{compact quantum group} is a pair~\(\G=\Bialg{A}\) consisting of a unital 
 \(\Cst\)\nb-algebra~\(A\) and a unital \Star{}homomorphism \(\Comult[A]\colon 
 A\to A\otimes A\) such that
 \begin{enumerate}
  \item \(\Comult[A]\) is coassociative: \((\Id_{A}\otimes\Comult[A])\circ\Comult[A] 
  =(\Comult[A]\otimes\Id_{A})\circ\Comult[A]\);
  \item \(\Comult[A]\) satisfies the cancellation conditions: \(\Comult[A](A)(1_{A}\otimes A)
  =A\otimes A=\Comult[A](A)(A\otimes 1_{A})\).
 \end{enumerate}
 Let~\(D\) be a unital~\(\Cst\)\nb-algebra. An action of a compact quantum group~\(\Qgrp{G}{A}\) on~\(D\) is a unital \Star{}homomorphism~\(\gamma\colon D\to 
 D\otimes A\) such that 
 \begin{enumerate}
  \item \(\gamma\) satisfies the action equation, that is, \((\Id_{D}\otimes\Comult[A])\circ\gamma=(\gamma\otimes\Id_{A})\circ\gamma\);
  \item \(\gamma\) satisfies the \emph{Podle\'s condition}: \(\gamma(D)(1_{D}\otimes A)
  =D\otimes A\).
 \end{enumerate}
 Suppose~\(\tau\colon D\to\C\) be a state. Then~\(\gamma\) is said to be~\emph{\(\tau\)\nb-preserving} if~\((\tau\otimes\Id_{A})\gamma(d)=\tau(d)\cdot 1_{A}\) for all~\(d\in D\). In fact, \(\gamma\) is injective whenever~\(\tau\) is faithful. 
 
 \begin{example}
  Represent~\(\Cont(\T)\) on~\(L^{2}(\T)\) faithfully by pointwise multiplication.  Let~\(z\in\Cont(\T)\) be the unitary operator of pointwise multiplication with the identity function on~\(\T\). Then \(\Comult[\Cont(\T)](z)=z\otimes z\) defines the comultiplication map 
  on~\(\Cont(\T)\). The pair~\(\Bialg{\Cont(\T)}\) is~\(\T\) viewed as a compact 
  quantum group. Suppose~\(D\) is a unital \(\Cst\)\nb-algebra and~\(\gamma\colon\T\to\Aut (D)\) is a continuous action. An element~\(d\in D\) is a homogeneous element of degree~\(\textup{deg}(d)\in\Z\) if~\(\gamma_{z}(d)=z^{\textup{deg}(d)}\cdot d\) for all~\(z\in\T\). Equivalently, 
\(\gamma(d)=d\otimes z^{\textup{deg}(d)}\) while~\(\gamma\) is viewed as 
an action of~\(\Bialg{\Cont(\T)}\) on~\(D\). 
 \end{example}

 \subsection{Braided tensor product of C*-algebras}
  \label{subsec:braidedtensor}
 Let~\(\Hils[L]_{1}\) and~\(\Hils[L]_{2}\) be separable Hilbert spaces. Let~\(\pi_{i}\colon\T\to\U(\Hils[L]_{i})\) be continuous representations for~\(i=1,2\). Suppose, 
 \(\{\lambda_{m}^{i}\}_{m\in\N}\) is an orthonormal basis of~\(\Hils[L]_{i}\) consisting of the eigenvectors for~\(\pi^{i}\), that is~\((\pi_{i})_{z}(\lambda_{m}^{i})=z^{l^{i}_{m}}\lambda_{m}^{i}\) for some~\(l_{m}^{i}\in\N\), for~\(i=1,2\). 
 
 By~\cite{MRW2016}*{Equation (3.2) \& Proposition 3.2}, the braiding unitary~\(\Braiding{\Hils[L]_{1}}{\Hils[L]_{2}}\colon\Hils[L]_{1}\otimes\Hils[L]_{2}\to\Hils[L]_{2}\otimes\Hils[L]_{1}\) and its inverse 
 \(\Dualbraiding{\Hils[L]_{2}}{\Hils[L]_{1}}\colon\Hils[L]_{2}\otimes\Hils[L]_{1}\to 
 \Hils[L]_{1}\otimes\Hils[L]_{2}\) associated to the \(\Rmat\)\nb-matrix~\eqref{eq:Rmat} are defined on the basis elements by
 \begin{equation}
  \label{eq:braiddef}
  \Braiding{\Hils[L]_{1}}{\Hils[L]_{2}}(\lambda_{m}^{1}\otimes\lambda_{n}^{2})
  =\zeta^{l_{m}^{1}l_{n}^{2}}(\lambda_{n}^{2}\otimes\lambda_{m}^{1}), 
  \qquad 
  \Dualbraiding{\Hils[L]_{2}}{\Hils[L]_{1}}(\lambda_{n}^{2}\otimes\lambda_{m}^{1})
  =\zeta^{-l_{m}^{1}l_{n}^{2}}(\lambda_{m}^{1}\otimes\lambda_{n}^{2}).
 \end{equation}
 Consider the category \(\Cstcat(\T)\) consisting of unital \(\Cst\)\nb-algebras with a continuous action of~\(\T\) as objects and \(\T\)\nb-equivariant unital~\Star{}homomorphisms as arrows. 
 
 Let~\((D_{i},\gamma_{i})\) are objects of~\(\Cstcat(\T)\) and~\(\psi_{i}\colon D_{i}\hookrightarrow \Bound(\Hils[L]_{i})\) be faithful \(\T\)\nb-equivariant representations 
 of~\((D_{i},\gamma_{i})\), respectively, for~\(i=1,2\). For any homogeneous elements 
 \(d_{i}\in D_{i}\) we have 
 \begin{align}
  \label{eq:cov-rep}
  & (\pi_{i})_{z}(d_{i}\lambda_{m}^{i})
  =\bigl((\psi_{i})_{z}(d_{i})\bigr) \bigl((\pi_{i})_{z}(\lambda_{m}^{i})\bigr)
  =z^{\text{deg}(d_{i})+l_{m}^{i}}d_{i}\lambda_{m}^{i},\\
  \label{eq:cov-rep-star}  
 & (\pi_{i})_{z}(d_{i}^{*}\lambda_{m}^{i})
  =\bigl((\psi_{i})_{z}(d_{i}^{*})\bigr) \bigl((\pi_{i})_{z}(\lambda_{m}^{i})\bigr)
  =z^{-\text{deg}(d_{i})+l_{m}^{i}}d_{i}\lambda_{m}^{i},
 \end{align}
 for all~\(z\in\T\) and~\(i=1,2\).
 
A unital \Star{}homomorphism \(f\colon D_{1}\to D_{2}\) is an arrow~\(D_{1}\to D_{2}\) in~\(\Cstcat(\T)\) if it is~\emph{\(\T\)\nb-equivariant}: \(f\circ (\gamma_1)_{z}=(\gamma_{1})_{z}\circ f\) for all~\(z\in\T\). The set of arrows~\(D_{1}\to D_{2}\) in~\(\Cstcat(\T)\) is denoted by~\(\Mor^{\T}(D_{1},D_{2})\). 

The~\(\Rmattxt\)\nb-matrix~\eqref{eq:Rmat} defines the \emph{braided tensor product}~\(D_{1}\boxtimes_{\zeta}D_{2}=j_{1}(D_{1})j_{2}(D_{2})\), 
where~\(j_{i}\colon D_{i}\to D_{1}\boxtimes_{\zeta}D_{2}\) are the canonical embeddings, for~\(i=1,2\). The concrete descriptions of these maps are given by 
\[
  j_{1}(d_{1})=\psi_{1}(d_{1})\otimes 1_{\Bound(\Hils[L]_{2})}, 
  \qquad 
  j_{2}(d_{2})=\Braiding{\Hils[L]_{1}}{\Hils[L]_{2}}
  (\psi_{2}(d_{2})\otimes 1_{\Bound(\Hils[L]_{1})})
  \Dualbraiding{\Hils[L]_{2}}{\Hils[L]_{1}},
\]
for all~\(d_{i}\in D_{i}\), where~\(i=1,2\). In particular, if \(d_{i}\in D_{i}\) are 
the homogeneous elements, then the commutation relation between the elementary tensor factors~\(j_{1}(d_{1})\defeq d_{1}\boxtimes_{\zeta}1_{D_{2}}\) and~\(j_{2}(d_{2})\defeq 1_{D_{1}}\boxtimes_{\zeta}d_{2}\) is given by
\begin{equation} 
 \label{eq:comm-brd_tensor}
    d_{1}\boxtimes_{\zeta}d_{2}
 \defeq (d_{1}\boxtimes_{\zeta} 1_{D_{2}})\cdot (1_{D_{1}}\boxtimes_{\zeta} d_{2}) 
 = \zeta^{-\textup{deg}(d_{1})\textup{deg}(d_{2})}
    (1_{D_{1}}\boxtimes_{\zeta} d_{2}) \cdot  (d_{1}\boxtimes_{\zeta} 1_{D_{2}}).
\end{equation}
The braided tensor product \(D_{1}\boxtimes_{\zeta}D_{2}\) carries the continuous diagonal action~\((\gamma_{1}\bowtie\gamma_{2})_{z}(d_{1}\boxtimes_{\zeta}d_{2})\defeq(\gamma_{1})_{z}(d_{1})\boxtimes_{\zeta}(\gamma_{2})_{z}(d_{2})\) for all~\(z\in\T\). Thus \((D_{1}\boxtimes_{\zeta}D_{2},\alpha\bowtie\beta)\) is an object of~\(\Cstcat(\T)\). Furthermore, \((\Cstcat(\T),\boxtimes_{\zeta})\) is a monoidal category with~\(\C\) along with the trivial action of~\(\T\) as monoidal unit. We refer to~\cites{MRW2014, MRW2016} for more details of this construction. In particular, \(\boxtimes_{1}\) coincides with the minimal tensor product~\(\otimes\) of~\(\Cst\)\nb-algebras. Also, 
\(A\boxtimes_{\zeta}B\cong A\otimes B\) if either~\(\alpha\) or~\(\beta\) is trivial.

\medskip 

 Suppose~\((B,\beta)\) and~\((D,\gamma)\) are objects of the category~\(\Cstcat(\T)\). 
By virtue of~\cite{MRW2016}*{Proposition 3.6}, there exists an injective unital~\Star{}homomorphism \(\Psi^{D,B}\colon \Cont(\T)\boxtimes_{\zeta} D\boxtimes_{\zeta}B\to 
  (\Cont(\T)\boxtimes_{\zeta}D)\otimes (\Cont(\T)\boxtimes_{\zeta}B)\). For~\(a\in\Cont(\T)\), \(d\in D\), \(b\in B\) it is defined by
  \begin{align}
   \label{eq:Psi}
   \begin{split}
    & \Psi^{D,B}(j_{1}(a)) =(j_{1}\otimes j_{1})\Comult[\Cont(\T)](a), 
    \qquad
    \Psi^{D,B}(j_{2}(b)) =(j_{2}\otimes j_{1})\gamma(d),\\
    &\Psi^{D,B}(j_{3}(b)) =1_{\Cont(\T)\boxtimes_{\zeta}D}\otimes j_{2}(b).
   \end{split}
  \end{align}
 Here we view the action~\(\beta\) as a unital \Star{}homomorphism~\(\beta\colon B\to B\otimes\Cont(\T)\). These maps will play crucial roles in Sections~\ref{sec:Boson} and~\ref{sec:univ-prop}.
\section{Action of braided quantum groups on matrix algebras}
 \label{sec:act_brd_cpt}
Fix~\(\zeta\in\T\) and recall the monoidal category~\((\Cstcat(\T),\boxtimes_{\zeta})\) from the previous section.
\begin{definition}[compare with~\cite{MRW2016}*{Definition 6.1}]
 \label{def:br_CQG}
A triple~\(\G=(A,\alpha,\Comult[A])\) is a \emph{braided compact quantum group}  in~\((\Cstcat,\boxtimes_{\zeta})\) if \((A,\alpha)\) is an object of~\(\Cstcat(\T)\) and~\(\Comult[A]\in\Mor^{\T}(A,A\boxtimes_{\zeta}A)\) such that
\begin{enumerate}
 \item \(\Comult[A]\) is coassociative 
  \begin{equation}
   \label{eq:comult_coasso}
   (\Comult[A]\boxtimes_{\zeta}\Id_{A})\circ\Comult[A]
 =(\Id_{A}\boxtimes_{\zeta}\Comult[A])\circ\Comult[A] ;
 \end{equation}
 \item \(\Comult[A]\) satisfies the cancellation conditions:
 \begin{equation}
  \label{eq:comult_cancel}
   \Comult[A](A) (1_{A}\boxtimes_{\zeta}A)=A\boxtimes_{\zeta}A=\Comult[A](A)(A\boxtimes_{\zeta}1_{A}).
 \end{equation}
\end{enumerate}
Then we say~\(\G=(A,\alpha,\Comult[A])\) is a \emph{braided compact quantum group over~\(\T\)}. We refer to~\cites{KMRW2016,MR2019a} for genuine examples of braided compact quantum groups over~\(\T\). Indeed, \(\G\) is an ordinary compact quantum group if either~\(\zeta=1\) or the action~\(\alpha\) is trivial.
\end{definition}
\begin{theorem}
 \label{the:qnt_aut_matrx}
 For a given natural number~\(n\), let~\(d_{1}\leq d_{2}\leq \cdots \leq d_{n}\) be integers. Let~\(A\) be the universal~\(\Cst\)\nb-algebra generated by~\((u^{ij}_{kl})_{1\leq i,j,k,l\leq n}\) subject to the following relations:
\begin{align}
 \label{eq:cond-2}
 &\sum_{t=1}^{n}(\zeta^{d_{t}(d_{k}-d_{i}+d_{t}-d_{s})}u^{it}_{ks})\cdot 
 (\zeta^{d_{j}(d_{m}-d_{t}+d_{j}-d_{l})}u^{tj}_{ml})\\ \nonumber
 &=\delta_{s,m}\cdot (\zeta^{d_{j}(d_{k}-d_{i}+d_{j}-d_{l}}u^{ij}_{kl}),
  \quad\text{for all~\(1\leq i,j,k,l,m,n\leq n\)\textup{;}} \\
 \label{eq:cond-3}
 & \sum_{t=1}^{n} (\zeta^{d_{s}(d_{k}-d_{i}+d_{s}-d_{t})}u^{is}_{kt})\cdot 
 (\zeta^{d_{j}(d_{t}-d_{m}+d_{j}-d_{l})}u^{mj}_{tl})\\ \nonumber
 &=\delta_{s,m}\cdot (\zeta^{d_{j}(d_{k}-d_{i}+d_{j}-d_{l})}u^{ij}_{kl}),
\quad \text{for all~\(1\leq i,j,k,l,m,n\leq n\)\textup{;}} \\
 \label{eq:cond-4}
  & u_{kl}^{ij}{ }^{*}
   = \zeta^{(d_{i}-d_{j})(d_{l}-d_{j}+d_{i}-d_{k})}(u^{ji}_{lk}), 
   \quad \text{for all~\(1\leq i,j,k,l\leq n\)\textup{;}} \\
  \label{eq:cond-5}
  & \sum_{r=1}^{n}u^{ij}_{rr}
  =\delta_{i,j},
  \quad \text{for all~\(1\leq i,j\leq n\)\textup{;}} \\
  \label{eq:cond-6}
  & \sum_{r=1}^{n}u^{rr}_{kl}
  =\delta_{k,l},
  \quad\text{for all~\(1\leq i,j\leq n\)\textup{.}}
\end{align}
There exists a unique continuous action~\(\alpha\) of~\(\T\) on~\(A\) satisfying
\begin{equation}
  \label{eq:cond-1}
  \alpha_{z}(u^{ij}_{kl}) 
    =z^{d_{k}-d_{i}+d_{j}-d_{l}}u^{ij}_{kl},
 \quad \text{for all~\(z\in\T\);}\\ 
\end{equation}
Thus, \((A,\alpha)\) is an object of~\(\Cstcat(\T)\). There exists a unique 
\(\Comult[A]\in\Mor^{\T}(A,A\boxtimes_{\zeta} A)\) such that 
\begin{equation}
 \label{eq:brd_comult}
 \Comult[A](u^{ij}_{kl})=\sum_{r,s=1}^{n}u^{ij}_{rs}\boxtimes_{\zeta} u^{rs}_{kl},
 \qquad\text{\(1\leq i,j,k,l\leq n\).}
\end{equation}
Moreover, \(\Comult[A]\) satisfies~\eqref{eq:comult_coasso}-\eqref{eq:comult_cancel}.
Then~\((A,\alpha,\Comult[A])\) is a braided compact quantum group over~\(\T\) denoted by~\(\Aut(\Mat{n}(\C),\Pi,\phi)\).
\end{theorem}
\begin{proof}
Let~\(W=\C^{n}\) and~\(\rho\) is a unitary representation of~\(\T\) on~\(W\). Let~\(\{e_1, e_{2}, \cdots , e_{n}\}\) be the eigenbasis for~\(\rho\) such that~\(\rho_{z}(e_i)=z^{d_{i}}e_{i}\) for all~\(1\leq i\leq n\).

We identify~\(\textup{End}(W)\) with~\(\Mat{n}(\C)\) with respect to this ordered basis. Then \(\Mat{n}(\C)\) is generated by the rank one matrices~\(\{E_{ij}\}_{1\leq i,j\leq n}\), defined by~\(E_{ij}e_{k}=\delta_{j,k}e_{i}\) for all~\(1\leq k\leq n\). In fact, 
\(\rho\) induces the action~\(\Pi\) in~\eqref{eq:T-act-Mat} of~\(\T\) on~\(\Mat{n}(\C)\). 

Similarly, we identify~\(\textup{End}(\Mat{n}(\C))\cong\Mat{n}(\C)\otimes\Mat{n}(\C)^{*}\cong \Mat{n}(\C)\otimes\Mat{n}(\C)\). Subsequently,~\(\Pi\) induces a continuous action~\(\widetilde{\Pi}\) of~\(\T\) 
on~\(\textup{End}(\Mat{n}(\C))\). On the basis vectors~\(E_{ij}\otimes E_{kl}\) it acts in the following way:
\begin{equation}
 \label{eq:T-act-Mat-tens-Mat}
   \widetilde{\Pi}_{z}(E_{ij}\otimes E_{kl})\defeq z^{d_{i}-d_{j}-d_{k}+d_{l}}E_{ij}\otimes E_{kl} \quad\text{for all \(z\in\T\).}
 \end{equation}
 Suppose~\(\mathcal{A}\) is the universal unital~\Star{}algebra generated by~\(u^{ij}_{kl}\) for~\(1\leq i,j\leq n\) satisfying~\eqref{eq:cond-1}-\eqref{eq:cond-6}. Define~\(u\in\Mat{n}(\C)\otimes\Mat{n}(\C)\otimes \mathcal{A}\) by
 \begin{equation}
  \label{eq:brd_fund-rep}
   u=\sum_{i,j,k,l=1}^{n}E_{ik}\otimes E_{jl}\otimes u^{ij}_{kl}.
 \end{equation}
 Since~\(\zeta\in\T\), we can rewrite~\eqref{eq:cond-2} using the formula~\eqref{eq:cond-1}  in the following compact form:
 \begin{equation}
  \label{eq:cond-2eqv}
   \sum_{t=1}^{n}\alpha_{\zeta^{d_{t}}}(u^{it}_{ks})\cdot 
   \alpha_{\zeta^{d_{j}}}(u^{tj}_{ml}) 
  =\delta_{s,m}\cdot \alpha_{\zeta^{d_{j}}}(u^{ij}_{kl}),
  \quad\text{for all~\(1\leq i,j,k,l,m,n\leq n\).}
 \end{equation}
 Then we use conditions~\eqref{eq:cond-4}, \eqref{eq:cond-2eqv}, \eqref{eq:cond-5} and the fact~\(\sum_{k=1}^{n}E_{kk}=1_{\Mat{n}(\C)}\) to show~\(u^{*}u=1_{\Mat{n}(\C)\otimes\Mat{n}(\C)}\otimes 1_{\mathcal{A}}\):
 \begin{align*}
     & \Bigl(\sum_{i,jk,l=1}^{n}E_{ik}\otimes E_{jl}\otimes u^{ij}_{kl}\Bigr)^{*} 
         \Bigl(\sum_{p,q,r,s=1}^{n}E_{pr}\otimes E_{qs}\otimes u^{pq}_{rs}\Bigr)\\
   &=\Bigl(\sum_{i,jk,l=1}^{n}E_{ki}\otimes E_{lj}\otimes u^{ij}_{kl}{ }^{*}\Bigr) 
        \Bigl(\sum_{p,q,r,s=1}^{n}E_{pr}\otimes E_{qs}\otimes u^{pq}_{rs}\Bigr)\\    
   &=\sum_{i,j,k,l,p,q,r,s=1}^{n} \delta_{i,p}\delta_{j,q} E_{kr}\otimes E_{ls}\otimes 
        \alpha_{\zeta^{d_{i}-d_{j}}}(u^{ji}_{lk})u^{pq}_{rs}\\
   &=\sum_{i,j,k,l,r,s=1}^{n} E_{kr}\otimes E_{ls}\otimes 
   \alpha_{\zeta^{d_{i}-d_{j}}}(u^{ji}_{lk})u^{ij}_{rs}\\
  &=\sum_{i,j,k,l,r,s=1}^{n} \zeta^{d_{j}(d_{k}-d_{l}+d_{s}-d_{r})}E_{kr}\otimes E_{ls}\otimes  
       \alpha_{\zeta^{d_{i}}}(u^{ji}_{lk})\alpha_{\zeta^{d_{j}}}(u^{ij}_{rs}) \\  
  &=\sum_{j,k,l,r,s=1}^{n} \delta_{k,r}\cdot \zeta^{d_{j}(d_{k}-d_{l}+d_{s}-d_{r})}E_{kr}\otimes E_{ls}\otimes \alpha_{\zeta^{d_{j}}}(u^{jj}_{ls})\\
  &=\sum_{j,k,l,r,s=1}^{n} \delta_{k,r}\cdot \zeta^{d_{j}(d_{k}-d_{r})}E_{kr}\otimes E_{ls}\otimes u^{jj}_{ls}\\
  &=\sum_{j,k,l,s=1}^{n} E_{kk}\otimes E_{ls}\otimes u^{jj}_{ls}
    =\sum_{j,l,s=1}^{n}1_{\Mat{n}(\C)}\otimes E_{ls}\otimes u_{ls}^{jj}\\
  &=1_{\Mat{n}(\C)\otimes\Mat{n}(\C)}\otimes 1_{\mathcal{A}}.
 \end{align*}
 Similarly, we can verify that~\(uu^{*}=1_{\Mat{n}(\C)\otimes\Mat{n}(\C)}\otimes 1_{\mathcal{A}}\). So, any \(\Cst\)\nb-seminorm~\(\norm{\cdot}\) on~\(\mathcal{A}\) will satisfy~\(\norm{u^{ij}_{kl}}\leq 1\) for all~\(1\leq i,j\leq n\) and~\(A\) is the completion of~\(\mathcal{A}\) with respect to the largest~\(\Cst\)\nb-seminorm on~\(\mathcal{A}\). 

Also, it is easy to verify that for all~\(z\in\T\) the elements~\((\alpha_{z}(u^{ij}_{kl}))_{1\leq i,j,k,l\leq n}\) in~\eqref{eq:cond-1} satisfy the relations~\eqref{eq:cond-2}-\eqref{eq:cond-6}. This ensures the existence and uniqueness of the action~\(\alpha\) of~\(\T\) on~\(A\).

Now we consider the diagonal action~\(\alpha\bowtie\alpha\) of~\(\T\) on~\(A\boxtimes_{\zeta}A\): ~\((\alpha\bowtie\alpha)_{z}(u^{ij}_{rs}\boxtimes_{\zeta}u^{rs}_{kl})=\alpha_{z}(u^{ij}_{rs})\boxtimes_{\zeta}\alpha_{z}(u^{rs}_{kl})=\zeta^{d_{i}-d_{j}+d_{k}-d_{l}}(u^{ij}_{rs}\boxtimes_{\zeta}u^{rs}_{kl})\) for all~\(1\leq r,s\leq n\). This shows that~\(\Comult[A](u^{ij}_{kl})\) is also a homogeneous element of~\(A\boxtimes_{\zeta}A\) with degree~\(d_{i}-d_{j}+d_{k}-d_{l}\). Hence~\(\Comult[A]\) is~\(\T\)\nb-equivariant on the generators of~\(A\). Next, we recall the commutation relation~\eqref{eq:comm-brd_tensor} and~\eqref{eq:cond-2eqv} to verify the latter for~\(\Comult[A](u^{ij}_{kl})\):
\begin{align*}
& \sum_{t=1}^{n}(\alpha\bowtie\alpha)_{\zeta^{d_{t}}}\bigl(\Comult[A](u^{it}_{ks})\bigr)
   \cdot (\alpha\bowtie\alpha)_{\zeta^{d_{j}}}\bigl(\Comult[A](u^{tj}_{ml})\bigr)\\
&= \sum_{t,p,q,x,y=1}^{n}
     \bigl(\alpha_{\zeta^{d_{t}}}(u^{it}_{pq})\boxtimes_{\zeta}
     \alpha_{\zeta^{d_{t}}}(u^{pq}_{ks})\bigr)
     \bigl(\alpha_{\zeta^{d_{j}}}(u^{tj}_{xy})\boxtimes_{\zeta}
     \alpha_{\zeta^{d_{j}}}(u^{xy}_{ml})\bigr)\\
&= \sum_{t,p,q,x,y=1}^{n}
     \zeta^{(d_{k}-d_{p}+d_{q}-d_{s})(d_{x}-d_{t}+d_{j}-d_{y})}\cdot
     \alpha_{\zeta^{d_{t}}}(u^{it}_{pq})\alpha_{\zeta^{d_{j}}}(u^{tj}_{xy})
     \boxtimes_{\zeta} 
     \alpha_{\zeta^{d_{t}}}(u^{pq}_{ks})\alpha_{\zeta^{d_{j}}}(u^{xy}_{ml})\\
&= \sum_{p,x,y=1}^{n}  
     \zeta^{(d_{k}-d_{p}+d_{x}-d_{s})(d_{x}-d_{t}+d_{j}-d_{y})}\cdot 
     \alpha_{\zeta^{d_{j}}}(u^{ij}_{py})\boxtimes_{\zeta}
      \alpha_{\zeta^{d_{t}}}(u^{px}_{ks})\alpha_{\zeta^{d_{j}}}(u^{xy}_{ml})\\
&= \sum_{p,x,y=1}^{n}  
     \zeta^{(d_{k}-d_{p}-d_{s}+d_{m}+d_{y}-d_{l})(d_{j}-d_{y})}\cdot      
     \alpha_{\zeta^{d_{j}}}(u^{ij}_{py})\boxtimes_{\zeta}
      \alpha_{\zeta^{d_{x}}}(u^{px}_{ks})\alpha_{\zeta^{d_{y}}}(u^{xy}_{ml})\\
&= \sum_{p,y=1}^{n}  
     \zeta^{(d_{k}-d_{p}-d_{s}+d_{m}+d_{y}-d_{l})(d_{j}-d_{y})}\cdot  
     \delta_{s,m} \cdot    
     \alpha_{\zeta^{d_{j}}}(u^{ij}_{py})\boxtimes_{\zeta}
      \alpha_{\zeta^{d_{y}}}(u^{py}_{kl})\\
&= \sum_{p,y=1}^{n}  
     \zeta^{(d_{m}-d_{s})(d_{j}-d_{y})}\cdot  
     \delta_{s,m} \cdot    
     \alpha_{\zeta^{d_{j}}}(u^{ij}_{py})\boxtimes_{\zeta}
      \alpha_{\zeta^{d_{j}}}(u^{py}_{kl})
   = \delta_{s,m} \cdot (\alpha\bowtie\alpha)_{\zeta^{d_{j}}}\bigl(u^{ij}_{kl}\bigr).
\end{align*}
Similarly, we can verify that~\eqref{eq:cond-3}-\eqref{eq:cond-6} for~\(\{\Comult[A](u_{kl}^{ij})\}_{i,j,k,l,=1}^{n}\). Then, by the universal property of~\(A\), \(\Comult[A]\) extends to a unique \(\T\)\nb-equivariant unital \Star{}homomorphism \(\Comult[A]\colon A\to A\boxtimes_{\zeta}A\). The coassociativity for~\(\Comult[A]\) is a routine check. 

In order to verify the cancellation conditions~\eqref{eq:comult_cancel} for~\(\Comult[A]\), we are going to employ the argument same as in~\cite{KMRW2016}*{Section 4}. Consider the set~\(S\defeq\{a\in A\mid a\boxtimes_{\zeta}1_{A}\in \Comult[A](A)(1\boxtimes_{\zeta}A)\}\). Recall the canonical inclusions~\(j_{1},j_{2}\colon A\rightrightarrows A\boxtimes_{\zeta}A\). Then we observe that 
\begin{equation}
 \label{eq:fund_unitary}
   (\Id_{\Mat{n^{2}}(\C)}\otimes\Comult[A])u 
   = \bigl((\Id_{\Mat{n^{2}}(\C)}\otimes j_{1})u\bigr)
      \bigl((\Id_{\Mat{n^{2}}(\C)}\otimes j_{2})u\bigr).
\end{equation}
Since~\(u\) is unitary, we have~\(u^{ij}_{kl}, u^{ij}_{kl}{ }^{*}\in S\) for all~\(1\leq i,j,k,l\leq n\). By virtue of~\cite{KMRW2016}*{Proposition 3.1}, for any two homogeneous elements~\(x,y\in S\) we have
\begin{align*}
  j_{1}(xy)
  =j_{1}(x)j_{1}(y)
  \in\Comult[A](A)j_{2}(A)j_{1}(y)
  &=\Comult[A](A)j_{1}(y)j_{2}(A)\\
  &\subseteq
  \Comult[A](A)\Comult[A](A)j_{2}(A)j_{2}(A)\\
  &=\Comult[A](A)j_{2}(A).
\end{align*}
 So,~\(xy\in S\). Therefore, all the monomials in~\(u^{ij}_{kl}\) and~\(u^{ij}_{kl}{ }^{*}\) belongs to~\(S\). Then~\(S\) is dense in~\(A\). Consequently, 
 \(A\boxtimes_{\zeta}A=j_{1}(A)j_{2}(A)\subseteq \Comult[A](A)j_{2}(A)\subseteq A\boxtimes_{\zeta}A\). Similarly, we can show~\(\Comult[A](A)j_{1}(A)=A\boxtimes_{\zeta}A\) by proving that \(\{a\in A\mid j_{2}(a)\in j_{1}(A)\Comult[A](A)\}\) is dense in~\(A\).
\end{proof}

\begin{definition}
 \label{def:br_act}
 An \emph{action} of a braided compact quantum group~\(\G=(B,\beta,\Comult[B])\)  over~\(\T\) on an object~\((D,\gamma)\) of~\(\Cstcat(\T)\) is an element 
 \(\eta\in\Mor^{\T}(D,D\boxtimes_{\zeta} B)\) such that
 \begin{enumerate}
  \item \(\eta\) is a comodule structure:  
   \begin{equation} 
   (\Id_{D}\boxtimes_{\zeta} \Comult[B])\circ \eta=(\eta\boxtimes_{\zeta}\Id_{B})\circ\eta
   \end{equation}
  \item \(\eta\) satisfies the \emph{Podle\'s condition}: \(\eta(D)(1_{D}\boxtimes_{\zeta} B)=D\boxtimes_{\zeta} B\).
 \end{enumerate}  
 Also, \(\eta\) is said to \emph{preserve} a state~\(f\in\Mor^{\T}(D,\C)\) if~\((f\boxtimes_{\zeta}\Id_{A})\eta(d)=f(d)1_{B}\) for all~\(d\in D\). 

 Consider faithful covariant representations of 
 \((D,\gamma)\) and~\((B,\beta)\) on some separable Hilbert spaces 
 \(\Hils[L]_{1}\) and \(\Hils[L]_{2}\) carrying~\(\T\)\nb-actions, respectively. Then \(\eta\) defines a faithful representation~\(\eta\colon D\to D\boxtimes_{\zeta}B\subseteq\Bound(\Hils[L]_{1}\otimes\Hils[L]_{2})\). 
 
 The action~\(\eta\) is called \emph{faithful} if the norm closure of the linear span of \(\{(\omega\otimes\Id_{\Bound(\Hils[L]_{2})})\eta(d)\mid \omega\in\Bound(\Hils[L]_{1})_{*}\}\subseteq\Bound(\Hils[L]_{2})\) coincides with image of~\(B\) inside~\(\Bound(\Hils[L]_{2})\).
 \end{definition}
 If~\(\eta\) preserves a faithful state~\(f\) on~\(D\), 
 then~\(\eta\) is injective.

 \begin{proposition}
  \label{prop:action_qnt_on_Matx}
  The map~\(\eta\colon\Mat{n}(\C)\to\Mat{n}(\C)\boxtimes_{\zeta} A\) defined by
  \[
    \eta(E_{ij})=\sum_{k,l=1}^{n}E_{kl}\boxtimes_{\zeta} u^{kl}_{ij} 
    \qquad\text{for all~\(1\leq i,j\leq n\),}
  \]
  is a faithful action of~\(\Aut(\Mat{n}(\C),\Pi,\phi)\) on~\((\Mat{n}(\C),\Pi)\) and preserves~\(\phi\).
 \end{proposition}
 \begin{proof}
 It is easy to verify that \(\eta\) is  \(\T\)\nb-equivariant, \(\eta(E_{ij}^{*})=\eta(E_{ij})^{*}\) and~\(\sum_{i=1}^{n}\eta(E_{ii})=1_{\Mat{n}(\C)}\boxtimes_{\zeta}1_{A}\) for all~\(1\leq i,j\leq n\). Also, using  the relation~\eqref{eq:cond-6} and the fact~\(\phi\) is~\(\T\)\nb-equivariant, we easily verify that~\(\eta\) preserves~\(\phi\).

To verify that~\(\eta\) is multiplicative, first we recall the commutation relation~\eqref{eq:comm-brd_tensor} and the condition~\eqref{eq:cond-2eqv}. Then for any~\(1\leq i,j,k,l\leq n\) we compute
  \begin{align*}
   &\eta(E_{ij})\eta(E_{kl})= \sum_{r,s,x,y=1}^{n} 
    	(E_{rs}\boxtimes_{\zeta}u^{rs}_{ij})
    	(E_{xy}\boxtimes_{\zeta}u^{xy}_{kl})\\
    &=\sum_{r,s,y=1}^{n}\zeta^{(d_{s}-d_{y})(d_{i}-d_{r}+d_{s}-d_{j})} 
    (E_{ry}\boxtimes_{\zeta}u^{rs}_{ij}u^{sy}_{kl})\\
    &=\sum_{r,s,y=1}^{n}\zeta^{d_{y}(d_{r}-d_{y}-d_{i}+d_{j}-d_{k}+d_{l})}E_{ry}\boxtimes_{\zeta}
     \bigl(\alpha_{\zeta^{d_{s}}}(u^{rs}_{ij})(\alpha_{\zeta^{d_{y}}}(u^{sy}_{kl})\bigr)\\
   &=\delta_{j,k}\sum_{r,y=1}^{n}\zeta^{d_{y}(d_{r}-d_{y}-d_{i}+d_{l})}E_{ry}\boxtimes_{\zeta} \alpha_{\zeta^{d_{y}}}(u^{ry}_{il})
   =\delta_{j,k}\cdot \eta(E_{il}).
  \end{align*}
  Thus~\(\eta\colon\Mat{n}(\C)\to\Mat{n}(\C)\boxtimes_{\zeta}A\) 
  is a well\nb-defined unital \Star{}homomorphism. Verification of Definition~\ref{def:br_act}~(1) is a routine check. Finally, consider the set~\(S\defeq\{ M\in\Mat{n}(\C)\mid M\boxtimes_{\zeta}1_{A}\in\eta(\Mat{n}(\C))(1_{\Mat{n}(\C)}\boxtimes_{\zeta} A)\}\). Since~\(u\) is unitary, in particular, we have~\(\sum_{r,s=1}^{n}u^{ij}_{rs}u^{rs}_{kl}{ }^{*}=\delta_{i,k}\cdot \delta_{j,s}\cdot 1_{A}\) for all~\(1\leq i,j,k,l\leq n\). This condition implies 
  \begin{align*}
     \sum_{r,s=1}^{n}\eta(E_{rs})(1\boxtimes_{\zeta}u^{rs}_{kl}{ }^{*})
   =\sum_{i,j,r,s=1}^{n}E_{ij}\boxtimes_{\zeta}(u^{ij}_{rs}u^{rs}_{kl}{ }^{*})
   &=\sum_{i,j=1}^{n}\delta_{i,k}\cdot\delta_{j,l}\cdot(E_{ij}\boxtimes_{\zeta} 1_{A})\\
   &=E_{kl}\boxtimes_{\zeta}1_{A},
  \end{align*}
  for all~\(1\leq k,l\leq n\). So~\(E_{kl},E_{kl}^{*}\in S\). Then using arguments similar to the proof of the first equality in~\eqref{eq:comult_cancel} as in Theorem~\ref{the:qnt_aut_matrx} we can verify Definition~\ref{def:br_act}~(2). 
  
   Let us consider faithful covariant representations of~\((\Mat{n}(\C),\Pi)\) and 
  \((A,\alpha)\) on some Hilbert spaces~\(\Hils[L]_{1}\) and~\(\Hils[L]_{2}\), respectively. 
  Then we fix an orthonormal basis~\(\{\lambda_{r}^{i}\}_{r\in\N}\) consisting of eigenvectors 
  for the representation of~\(\T\) on~\(\Hils[L]_{i}\), for~\(i=1,2\). We fix~\(i,j\in 
  \{1,\cdots ,n\}\). Then we use 
  \eqref{eq:braiddef}-\eqref{eq:cov-rep-star} and compute
  \[
    \eta(E_{ij})(\lambda_{r}^{1}\otimes\lambda_{s}^{2})
    =\sum_{k,l=1}^{n}
    \zeta^{l^{1}_{r}\textup{deg}(u^{kl}_{ij})} (E_{kl}\lambda_{r}^{1}
    \otimes u^{kl}_{ij}\lambda^{2}_{s}),
    \qquad\text{where~\(r,s\in\N\).} 
  \]
  Next we fix~\(a,b\in\{1,\cdots ,n\}\). Define~\(f\colon\{1,\cdots ,n\}\to\C\) by 
  \(f(y)=\zeta^{-\textup{deg}(E_{by})\textup{deg}(u^{ab}_{ij})}\). For any 
  \(\xi_{1},\xi_{2}\in\Hils[L]_{1}\), the vector functionals \(\omega_{\xi_{1},\xi_{2}}\) are defined by~\(\omega_{\xi_{1},\xi_{2}}(T)=\langle \xi_{1},T(\xi_{2})\rangle\) for all~\(T\in\Bound(\Hils[L]_{1})\). For a fixed~\(r\in\N\), define \(\omega\defeq \sum_{y=1}^{n}f(y)\omega_{E_{ay}\lambda_{r}^{1},E_{by}
  \lambda_{r}^{1}}\). A simple computation using the last equation gives
  \[
   \bigl((\omega\otimes\Id_{\Bound(\Hils[L]_{2})})\eta(E_{ij})\bigr)\lambda_{s}^{2}\\
   =\left\langle \sum_{y=1}^{n} \lambda_{r}^{1},E_{yy}\lambda_{r}^{1}\right\rangle
  \zeta^{l_{r}^{1}\textup{deg}(u^{ab}_{ij})}u^{ab}_{ij}\lambda_{s}^{2}
  =\zeta^{l_{r}^{1}\textup{deg}(u^{ab}_{ij})}u^{ab}_{ij}\lambda_{s}^{2}.\qedhere
  \]
\end{proof}

 \section{The bosonisation} 
  \label{sec:Boson}
  The \emph{bosonisation} of a braided compact quantum group over~\(\T\) 
  is an ordinary compact quantum group together with an 
  idempotent quantum group homomorphism (``projection'') with image~\(\Bialg{\Cont(\T)}\). Conversely, every compact quantum group with projection having~\(\Bialg{\Cont(\T)}\) as image is the bosonisation of a braided compact quantum group over~\(\T\)~\cite{MRW2016}*{Theorem 6.4}. 
  We construct the bosonisation of~\(\Aut(\Mat{n}(\C),\Pi,\phi)\) in the following theorem.
  \begin{theorem}
   \label{the:Boson}
   Suppose~\(\Bialg{C}\) is the bosonisation of 
   \(\Aut(\Mat{n}(\C),\Pi,\phi)\). Then~\(C\) is isomorphic to the universal~\(\Cst\)\nb-algebra generated by the elements~\(z\) 
   and~\(u_{ij}^{kl}\) for~\(1\leq i,j,k,l\leq n\) subject to the relations~\(z^{*}z=zz^{*}=1\), \eqref{eq:cond-2}-\eqref{eq:cond-6} and~\(zu^{ij}_{kl}z^{*}=\zeta^{d_{i}-d_{k}+d_{l}-d_{j}}u^{ij}_{kl}\). The comultiplication map~\(\Comult[C]\) is given by 
   \begin{equation}
    \label{eq:comult-bos}
    \Comult[C](z)=z\otimes z, 
    \qquad 
    \Comult[C](u^{ij}_{kl})=\sum_{r,s=1}^{n}u^{ij}_{rs}\otimes z^{d_{r}-d_{i}+d_{j}-d_{s}}u^{rs}_{kl}.
  \end{equation}
  In particular, \(C\cong A\rtimes\Z\), where~\(\alpha_{\zeta^{-1}}\) generates the action of~\(\Z\) on~\(A\).
  We denote~\(\Bialg{C}\) by~\(\textup{Bos}(\Aut(\Mat{n}(\C),\Pi,\phi))\).
  \end{theorem}  
  \begin{proof}
  Let~\(\Hils[L]\) be a separable Hilbert space with a continuous representation~\(\pi\) 
 of~\(\T\). Suppose~\(\{\lambda_{m}\}_{m\in\N}\) is an orthonormal basis of~\(\Hils[L]\) consisting of the eigenvectors for the~\(\T\)\nb-action, that is, \(\pi_{z}(\lambda_{m})=z^{l_{m}}\lambda_{m}\) for some~\(l_{m}\in\N\). 
The braiding unitary~\(\Braiding{\Hils[L]}{L^{2}(\T)}\colon\Hils[L]\otimes L^{2}(\T)\to L^{2}(\T)\otimes\Hils[L]\) and its inverse \(\Dualbraiding{L^{2}(\T)}{\Hils[L]}\colon L^{2}(\T)\otimes\Hils[L]\to\Hils[L]\otimes L^{2}(\T)\) are similar to~\eqref{eq:braiddef}. On the 
basis vectors they are defined by
 \begin{equation*} 
  \Braiding{\Hils[L]}{L^{2}(\T)}(\lambda_{m}\otimes z^{r})=\zeta^{rl_{m}} (z^{r}\otimes\lambda_{m}), 
  \qquad
  \Dualbraiding{L^{2}(\T)}{\Hils[L]}(z^{r}\otimes \lambda_{m})=\zeta^{-rl_{m}} (\lambda_{m}\otimes z^{r}). 
 \end{equation*}
 These are the same braiding unitaries in~\cite{MRW2016}*{Section 6}. Here~\(A\) is viewed as a~\(\T\)\nb-Yetter\nb-Drinfeld \(\Cst\)\nb-algebra with respect to the \(\T\)\nb-action~\(\alpha\) and the~\(\Z\)\nb-action 
 generated by~\(\alpha_{\zeta^{-1}}\). The~\(\Z\)\nb-action is the composition of~\(\alpha\) with the group homomorphism induced by the~\(\Rmat\)\nb-matrix~\eqref{eq:Rmat}, \(\Z\to\T\), \(n\to\zeta^{-n}\) as observed in~\cite{RR2021}*{Appendix A}. 
 
 Consider a faithful \(\T\)\nb-equivariant representation~\(A\hookrightarrow\Bound(\Hils[L])\). Then~\eqref{eq:cov-rep} and~\eqref{eq:cov-rep-star} give
 \[
   \pi_{z}(u^{ij}_{kl}\lambda_{m})
     =z^{d_{k}-d_{i}+d_{j}-d_{l}+l_{m}}u^{ij}_{kl}\lambda_{m},
  \qquad
   \pi_{z}(u^{ij}_{kl}{ }^{*}\lambda_{m})
  =z^{d_{i}-d_{k}+d_{l}-d_{j}+l_{m}}u^{ij}_{kl}{ }^{*}\lambda_{m}.  
 \]
 
 Therefore, the underlying \(\Cst\)\nb-algebra~\(C\) is the braided tensor product:
 \[
   C=\Cont(\T)\boxtimes_{\zeta}A=(\Cont(\T)\otimes 1_{\Bound(\Hils[L])})\cdot \Braiding{\Hils[L]}{L^{2}(\T)}(A\otimes 1_{\Bound(L^{2}(\T))})\Dualbraiding{L^{2}(\T)}{\Hils[L]}\subseteq\Bound(L^{2}(\T)\otimes\Hils[L]).
 \]
 In fact, \(C\) is generated by~\(z\boxtimes_{\zeta} 1_{A}=z\otimes 1\) and \(1_{\Cont(\T)}\boxtimes_{\zeta}u^{ij}_{kl}=
 \Braiding{\Hils[L]}{L^{2}(\T)}\cdot (u^{ij}_{kl}\otimes 1_{\Bound(L^{2}(\T))})\cdot \Dualbraiding{L^{2}(\T)}{\Hils[L]}\). The latter 
 operator acts on~\(L^{2}(\T)\otimes\Hils[L]\) by 
 \[
  \Braiding{\Hils[L]}{L^{2}(\T)}\cdot (u^{ij}_{kl}\otimes 1_{\Bound(L^{2}(\T))})\cdot \Dualbraiding{L^{2}(\T)}{\Hils[L]}(z^{r}\otimes \lambda_{m})
  =\zeta^{r(d_{k}-d_{i}+d_{l}-d_{j})}(z^{r}\otimes u^{ij}_{kl}\lambda_{m}),
 \]
 and consequently we get the commutation relation~\ref{eq:comm-brd_tensor}:
 \[
   (z\boxtimes_{\zeta}1_{A})(1_{\Cont(\T)}\boxtimes_{\zeta} u^{ij}_{kl})(z^{*}\boxtimes_{\zeta}1_{A}) 
   =\zeta^{-(d_{k}-d_{i}+d_{j}-d_{l})} (1_{\Cont(\T)}\boxtimes_{\zeta}u^{ij}_{kl})
   =1_{\Cont(\T)}\boxtimes_{\zeta} \alpha_{\zeta^{-1}}(u^{ij}_{kl}).
 \]
 The above equation shows that the unitaries~\(z^{r}\boxtimes_{\zeta}1_{A}\) and the faithful representation~\(A\hookrightarrow \Cont(\T)\boxtimes_{\zeta}A\subseteq\Bound(L^{2}(\T)\otimes\Hils[L])\) form a covariant representation for the~\(\Z\)\nb-action on~\(A\) generated by~\(\alpha_{\zeta^{-1}}\). Moreover, it is unitarily equivalent to the regular representation that defines the reduced crossed product for the~\(\Z\)\nb-action on~\(A\) generated by~\(\alpha_{\zeta^{-1}}\). So, \(C=\Cont(\T)\boxtimes_{\zeta}A\cong A\rtimes\Z\). Then, \(C\) is the universal \(\Cst\)\nb-algebra generated by a unitary~\(z\) and~\(\{u^{ij}_{kl}\}_{1\leq i,j,k,l\leq n}\) subject to the relations~\eqref{eq:cond-2}-\eqref{eq:cond-6} and the commutation relation
 \[
   zu^{ij}_{kl}=\zeta^{d_{i}-d_{k}+d_{l}-d_{j}}u^{ij}_{kl}z.
 \]
The comultiplication map~\(\Comult[C]\colon C\to C\otimes C\) is defined by~\(\Comult[C]
=\Psi^{A,A}\circ (\Id_{\Cont(\T)}\boxtimes_{\zeta}\Comult[A])\), where~\(\Psi^{A,A}\colon \Cont(\T)\boxtimes_{\zeta} A\boxtimes_{\zeta} A\to (\Cont(\T)\boxtimes_{\zeta} A)\otimes (\Cont(\T)\boxtimes_{\zeta} A)=C\otimes C\) is given by~\eqref{eq:Psi} for~\(B=D=A\) and~\(\beta=\gamma=\alpha\).
Using~\(\Comult[\Cont(\T)](z)=z\otimes z\) and \eqref{eq:brd_comult}, we compute
\begin{align*}
 \Comult[C](j_{1}(z)) &=(j_{1}\otimes j_{1})\Comult[\Cont(\T)](z)=j_{1}(z)\otimes j_{1}(z),\\
 \Comult[C](j_{2}(u^{ij}_{kl})) 
 &= \Psi^{A,A}\Bigl(\sum_{r,s=1}^{n}j_{2}(u^{ij}_{rs})j_{3}(u^{rs}_{kl}) \Bigr)\\
 &=\sum_{r,s=1}^{n}\bigl(j_{2}(u^{ij}_{rs})\otimes j_{1}(z^{d_{r}-d_{i}+d_{j}-d_{s}})\bigr)\cdot \bigl(1_{C}\otimes j_{2}(u^{rs}_{kl})\bigr)\\
 &=\sum_{r,s=1}^{n} j_{2}(u^{ij}_{rs})\otimes j_{1}(z^{d_{r}-d_{i}+d_{j}-d_{s}})j_{2}(u^{rs}_{kl}).
\end{align*}
After dropping the inclusion maps~\(j_{1}\) and~\(j_{2}\) from our notation above we get the formulas for~\(\Comult[C]\).
 \end{proof}
 
\section{The universal property of~\(\textup{Aut}(\textup{M}_{n}(\C),\Pi,\phi)\) and its bosonisation}
 \label{sec:univ-prop}
 Let~\((D,\gamma)\) be an object of the category~\(\Cstcat(\T)\). Let~\(\G=(B,\beta,\Comult[B])\) be a braided compact quantum group over~\(\T\) and let \(\Bialg{C'}\) be its bosonisation. Suppose \(\eta\in\Mor^{\T}(D,D\boxtimes_{\zeta}B)\) is an action of~\(\G\) on~\(D\).  Then~\(C'=\Cont(\T)\boxtimes_{\zeta}B\) and the comultiplication map~\(\Comult[C']=\Psi^{B,B}\circ(\Id_{\Cont(\T)}\boxtimes_{\zeta}\Comult[B])\), where~\(\Psi^{B,B}\) is given by~\eqref{eq:Psi} for~\((D,\gamma)=(B,\beta)\).

\begin{proposition}
 \label{prop:bos_act_lift}
 Define~\(\widetilde{\eta}\colon \Cont(\T)\boxtimes_{\zeta}D\to(\Cont(\T)\boxtimes_{\zeta}D)\otimes (\Cont(\T)\boxtimes_{\zeta}B)=(\Cont(\T)\boxtimes_{\zeta}D)\otimes C'\) by~\(\widetilde{\eta}=\Psi^{D,B}\circ(\Id_{\Cont(\T)}\boxtimes_{\zeta}\eta)\), where~\(\Psi^{D,B}\) is given by\textup{~\eqref{eq:Psi}}. Then~\(\widetilde{\eta}\) is an action of the ordinary compact quantum group \(\Bialg{C'}\) 
 on~\(\Cont(\T)\boxtimes_{\zeta}D\). 
\end{proposition}
\begin{proof}

Since~\(\Psi^{D,B}\) and~\(\eta\)  are injective, so is~\(\widetilde{\eta}\). Next, we consider~\((D,\gamma)\), \((B\boxtimes_{\zeta}B,\beta\bowtie\beta)\) and~\((D\boxtimes_{\zeta}B,\gamma\bowtie\beta)\), \((B,\beta)\) in~\eqref{eq:Psi}. Then we get the following injective unital \Star{}homomorphisms:
\begin{align*}
  & \Psi^{D,B\boxtimes_{\zeta}B}\colon \Cont(\T)\boxtimes_{\zeta}D\boxtimes_{\zeta}(B\boxtimes_{\zeta}B) 
  \to (\Cont(\T)\boxtimes_{\zeta}D)\otimes (\Cont(\T)\boxtimes_{\zeta}B\boxtimes_{\zeta}B),\\
  & \Psi^{D\boxtimes{\zeta}B,B}\colon \Cont(\T)\boxtimes_{\zeta}(D\boxtimes_{\zeta}B)\boxtimes_{\zeta}B
  \to (\Cont(\T)\boxtimes_{\zeta}D\boxtimes_{\zeta}B)\otimes (\Cont(\T)\boxtimes_{\zeta}B),
 \end{align*}
satisfying
 \begin{alignat*}{3}
  &\Psi^{D,B\boxtimes{\zeta}B} j_{1}(a)
  =(j_{1}\otimes j_{1})\Comult[\Cont(\T)](a),
  \qquad
  &&\Psi^{D\boxtimes{\zeta}B,B} j_{1}(a)
  =(j_{1}\otimes j_{1})\Comult[\Cont(\T)](a),\\     
  &\Psi^{D,B\boxtimes{\zeta}B} j_{2}(d) 
  =(j_{2}\otimes j_{1})\gamma(d),
  \qquad
  &&\Psi^{D\boxtimes{\zeta}B,B} j_{2}(d)    
  =(j_{2}\otimes j_{1})\gamma(d),\\
  &\Psi^{D,B\boxtimes{\zeta}B} j_{3}(b)
  =1_{\Cont(\T)\boxtimes_{\zeta}D}\otimes j_{2}(b),
  \qquad
  &&\Psi^{D\boxtimes{\zeta}B,B} j_{3}(b)
  =(j_{3}\otimes j_{1})\beta(b),\\
  &\Psi^{D,B\boxtimes{\zeta}B} j_{4}(b)
  =1_{\Cont(\T)\boxtimes_{\zeta}D}\otimes j_{3}(b), 
  \qquad
  &&\Psi^{D\boxtimes{\zeta}B,B} j_{4}(b)
  =1_{\Cont(\T)\boxtimes_{\zeta}D\boxtimes_{\zeta}B}\otimes j_{2}(b),
 \end{alignat*}
 for all \(a\in\Cont(\T)\),~\(b\in B\) and~\(d\in D\). Using these maps and the equivariance of \(\Comult[B]\) with respect to the \(\T\)\nb-actions~\(\beta\) and~\(\beta\bowtie\beta\) on~\(B\) and~\(B\boxtimes_{\zeta}B\), we obtain the following identities:
 \begin{align}
  \label{eq:aux-1}
  \begin{split}
  (\Id_{\Cont(\T)\boxtimes_{\zeta}D}\otimes(\Id_{\Cont(\T)}\boxtimes_{\zeta}\Comult[B]))\circ \Psi^{D,B}
  &= \Psi^{D,B\boxtimes_{\zeta}B}\circ (\Id_{\Cont(\T)\boxtimes_{\zeta}D}\boxtimes_{\zeta}\Comult[B]),\\
  ((\Id_{\Cont(\T)}\boxtimes_{\zeta}\eta)\otimes\Id_{\Cont(\T)\boxtimes_{\zeta}B})\circ\Psi^{D,B}
  &=\Psi^{D\boxtimes_{\zeta}B,B}\circ (\Id_{\Cont(\T)}\boxtimes_{\zeta}\eta\boxtimes_{\zeta}\Id_{B}).
  \end{split}
 \end{align}
 Using the condition~\((\gamma\otimes\Id_{\Cont(\T)})\circ\gamma = (\Id_{D}\otimes\Comult[\Cont(\T)]) \circ \gamma\) we can verify the following equation~\((\Id_{\Cont(\T)\boxtimes_{\zeta}D}\otimes\Psi^{B,B})\circ\Psi^{D,B\boxtimes_{\zeta}B}\circ j_{l}=(\Psi^{D,B}\otimes\Id_{\Cont(\T)\boxtimes_{\zeta}B})\circ\Psi^{D\boxtimes_{\zeta}B,B}\circ j_{l}\) for~\(l=1,2,3,4\). Combining it with~\eqref{eq:aux-1} and using~\eqref{def:br_act} we get 
  \[
  (\widetilde{\eta}\otimes\Id_{C'})\circ\widetilde{\eta}\circ j_{l} =
  (\Id_{\Cont(\T)\boxtimes_{\zeta}D}\otimes\Comult[C'])\circ\widetilde{\eta}\circ j_{l}
 \qquad
 \text{for~\(l=1,2\).}
  \]
  Recall the (braided) Podle\'s condition Definition~\ref{def:br_act}(2) for~\(\eta\), the (ordinary) Podle\'s condition~\(\gamma(D)(1_{D}\otimes\Cont(\T)=D\otimes\Cont(\T)\) for~\(\gamma\) and the cancellation condition~\(\Comult[\Cont(\T)](\Cont(\T)(1\otimes\Cont(\T))=\Cont(\T)\otimes\Cont(\T)\) for~\(\Comult[\Cont(\T)]\). Using them 
  we verify the Podle\'s condition for~\(\widetilde{\eta}\):
  \begin{align*}
   & \widetilde{\eta}\bigl(j_{1}(\Cont(\T))j_{2}(D)\bigr)
       \bigl(1_{\Cont(\T)\boxtimes_{\zeta}D}\otimes C')\bigr)\\
   &=\bigl(\Psi^{D,B}(\Cont(\T)\boxtimes_{\zeta}\eta(D))\bigr)
   \bigl(1_{\Cont(\T)\boxtimes_{\zeta}D}\otimes j_{2}(B)j_{1}(\Cont(\T))\bigr)\\ 
   &=\bigl(\Psi^{D,B}(\Cont(\T)\boxtimes_{\zeta}(\eta(D)j_{2}(B))\bigr)
   \bigl(1_{\Cont(\T)\boxtimes_{\zeta}D}\otimes j_{1}(\Cont(\T))\bigr)\\ 
   &=\bigl(\Psi^{D,B}(\Cont(\T)\boxtimes_{\zeta}D\boxtimes_{\zeta}B)\bigr)
   \bigl(1_{\Cont(\T)\boxtimes_{\zeta}D}\otimes j_{1}(\Cont(\T))\bigr)\\ 
   &=\bigl((j_{1}\otimes j_{1})\Comult[\Cont(\T)](\Cont(\T))\bigr)
        \bigl((j_{2}\otimes j_{1})\gamma(D)\bigr)
       \bigl(1_{\Cont(\T)\boxtimes_{\zeta}D}\otimes j_{2}(B)j_{1}(\Cont(\T))\bigr)\\         
   &=\bigl((j_{1}\otimes j_{1})\Comult[\Cont(\T)](\Cont(\T))\bigr)
        \bigl((j_{2}\otimes j_{1})(\gamma(D)(1\otimes\Cont(\T))\bigr)
       \bigl(1_{\Cont(\T)\boxtimes_{\zeta}D}\otimes j_{2}(B)\bigr)\\       
   &=\bigl((j_{1}\otimes j_{1})(\Comult[\Cont(\T)](\Cont(\T))(1_{\Cont(\T)}\otimes\Cont(\T))\bigr)
       \bigl(j_{2}(D)\otimes j_{2}(B)\bigr)\\
   &= j_{1}(\Cont(\T))j_{2}(D)\otimes j_{1}(\Cont(\T))j_{2}(B)
     =\Cont(\T)\boxtimes_{\zeta}D\otimes C' . \qedhere
  \end{align*}
\end{proof}
\subsection{Universal property of $\textup{Aut}(\textup{M}_{n}(\mathbb{C}),\Pi,\phi)$}
\begin{definition}
Let~\(\G_{1}=(B_{1},\beta_{1},\Comult[B_{1}])\) and~\(\G_{2}=(B_{2},\beta_{2},\Comult[B_{2}])\) be braided compact quantum groups over~\(\T\). A ~\emph{braided compact quantum groups homomorphism} \(f\colon\G_1\to \G_2\) is an element~\(f\in\Mor^{\T}(B_{1},B_{2})\) such that~\((f\boxtimes_{\zeta}f)\circ\Comult[B_{1}]=\Comult[B_{2}]\circ f\).
\end{definition}
Let~\(D\) be a unital \(\Cst\)\nb-algebra carrying an action~\(\gamma\) of~\(\T\). Suppose 
\(\gamma\) preserves a faithful state~\(\phi\) on~\(D\). Let~\(\mathcal{C}^{\T}(D,\gamma,\phi)\) be the category with the pairs~\((\G,\eta)\) consisting of a braided compact quantum groups~\(\G\) over~\(\T\) and \(\eta\) is a \(\phi\) preserving action of~\(\G\) on~\((D,\gamma)\) as objects. An arrow between two objects~\((\G_{1},\eta_{1})\) and~\((\G_{2},\eta_{2})\) is a braided compact quantum group homomrphism~\(f\colon\G_{1}\to\G_{2}\) such that~\((\Id_{\Mat{n}(\C)}\boxtimes_{\zeta}f)\circ \eta_{1}=\eta_{2}\circ f\). 
\begin{theorem}
 \label{the:univ_obj}
  \(\Aut(\Mat{n}(\C),\Pi,\phi)\) is the universal, initial object of the category \(\mathcal{C}^{\T}(\Mat{n}(\C),\Pi,\phi)\).
\end{theorem}

 Let~\(\G=(B,\beta,\Comult[B])\) be a braided compact quantum group over~\(\T\) and let \(\eta\colon \Mat{n}(\C)\to\Mat{n}(\C)\boxtimes_{\zeta} B\) be an action of~\(\G\) on~\(\Mat{n}(\C)\).
 Since~\(\{E_{ij}\}_{i,j=1}^{n}\) is basis of~\(\Mat{n}(\C)\), ~\(\eta\) is uniquely determined by its matrix coefficients~\(\{b_{ij}^{kl}\}_{1\leq i,j,k,l\leq n}\subseteq B\) given by~\eqref{eq:Mn-action}.

Now~\(\eta\) is~\(\T\)\nb-equivariant for the actions~\(\Pi\) and~\(\Pi\bowtie\beta\), and 
\(E_{ij}\in\Mat{n}(\C)\) is a homogeneous element of degree~\(d_{i}-d_{j}\). Then 
\(\eta(E_{ij})\in\Mat{n}(\C)\boxtimes_{\zeta}B\) must also be a homogeneous element of degree~\(d_{i}-d_{j}\). Therefore, each term on the right hand side of~\eqref{eq:Mn-action} appearing under summation must also be a homogeneous element of degree~\(d_{i}-d_{j}\). Equivalently,
 \[
       z^{d_{i}-d_{j}}(E_{kl}\boxtimes_{\zeta} b^{kl}_{ij})
    = (\Pi\bowtie\beta)_{z} (E_{kl}\boxtimes_{\zeta} b^{kl}_{ij})
    =\Pi_{z}(E_{kl})\boxtimes_{\zeta}\beta_{z}(b^{kl}_{ij})
    =z^{d_{k}-d_{l}}E_{kl}\boxtimes_{\zeta}\beta_{z}(b^{kl}_{ij})
 \]
Therefore, the restriction of~\(\beta\) on the elements~\(\{b^{ij}_{kl}\}_{i,j,k,l=1}^{n}\subseteq B\)
 is given by 
  \begin{equation}
  \label{eq:deg_b}
  \beta_{z}(b^{ij}_{kl})=z^{d_{k}-d_{i}+d_{j}-d_{l}}b^{ij}_{kl},
  \quad\text{for all~\(z\in\T\).}
 \end{equation}
Next we determine the restriction of~\(\Comult[B]\) on~\(\{b^{ij}_{kl}\}_{i,j,k,l=1}^{n}\). Let~\(\Hils[L]'\) be a separable Hilbert space with a continuous representation~\(\pi'\) of~\(\T\). Suppose~\(\{\lambda'_{m}\}_{m\in\N}\) is an orthonormal basis of eigenvectors for the~\(\T\)\nb-action: \(\pi'_{z}(\lambda_{m})=z^{l'_{m}}\lambda'_{m}\) for some~\(l'_{m}\in\N\). Let~\(B\hookrightarrow\Bound(\Hils[L])\) be a faithful,~\(\T\)\nb-equivariant representation. 

 Also, recall the \(n\)\nb-dimensional vector space~\(W\) with the basis~\(\{e_{1},e_{2},\cdots ,e_{n}\}\) equipped with \(\T\)\nb-action~\(\rho\) considered in the proof of the Theorem~\ref{the:qnt_aut_matrx}. Then \(E_{ij}e_{k}=\delta_{j,k}e_{i}\) for all~\(1\leq i,j,k\leq n\) is a faithful and~\(\T\)\nb-equivariant representation of~\(\Mat{n}(\C)\) on~\(W\).

The braiding unitary~\(\Braiding{\Hils[L]'}{W}\colon\Hils[L]'\otimes W\to W\otimes\Hils[L]'\) and its inverse \(\Dualbraiding{W}{\Hils[L]'}\colon W\otimes\Hils[L]'\to\Hils[L]'\otimes W\) are similar to~\eqref{eq:braiddef}. On the basis vectors they are defined by 
 \begin{equation} 
  \label{eq:braiddef-1}
  \Braiding{\Hils[L]'}{W}(\lambda'_{m}\otimes e_{i})=\zeta^{d_{i}l'_{m}}\cdot (e_{i}\otimes\lambda'_{m}), 
  \qquad
  \Dualbraiding{W}{\Hils[L]'}(e_{i}\otimes \lambda'_{m})=\zeta^{-d_{i}l'_{m}} \cdot (\lambda'_{m}\otimes e_{i}). 
 \end{equation}
The elementary braided tensors of~\(\Mat{n}(\C)\boxtimes_{\zeta} B\) are defined by
 \[
   a\boxtimes 1_{B}=a\otimes 1_{\Bound(\Hils[L]')}, 
   \qquad 
   1_{\Mat{n}(\C)}\boxtimes b=\Braiding{\Hils[L]'}{W}(b\otimes 1_{\Bound(W)})\Dualbraiding{W}{\Hils[L]'}.
 \]
 In particular, we note that  
 \begin{equation*}
   (1_{\Mat{n}(\C)}\boxtimes_{\zeta} b^{ij}_{kl}) (e_{k}\otimes\lambda'_{m})
   =\zeta^{-d_{k}l'_{m}}\cdot \Braiding{\Hils[L]'}{W}(b^{ij}_{kl}\lambda'_{m}\otimes e_{k})
   =\zeta^{d_{k}(d_k-d_i+d_j-d_l)} e_{k}\otimes b^{ij}_{kl}\lambda'_{m},
  \end{equation*}
  for all~\(1\leq k\leq n\) and~\(m\in\N\). It also gives us the expected commutation 
  relation~\eqref{eq:comm-brd_tensor}:
  \begin{equation}
  \label{eq:E-b-comm}
  (1\boxtimes b^{ij}_{kl})(E_{rs}\boxtimes 1)
   =\zeta^{(d_{r}-d_{s})(d_{k}-d_{i}+d_{j}-d_{l})} (E_{rs}\boxtimes b^{ij}_{kl}),
   \quad\text{for~\(1\leq i,j,k,l,r,s\leq n\).}
 \end{equation}
 The second condition in Definition~\ref{def:br_act} implies
 \[
   \sum_{k,l,r,s=1}^{n}E_{rs}\boxtimes b^{rs}_{kl}\boxtimes b^{kl}_{ij} 
   =(\alpha\boxtimes\Id_{B})\alpha(E_{ij}) 
   =\sum_{r,s=1}^{n}E_{rs}\boxtimes\Comult[B](b^{rs}_{ij}).
 \]
 Next we compute the action of the first and the third terms appearing 
 in the above equation on the basis vectors of~\(W\otimes\Hils[L]'\otimes\Hils[L]'\). 
 For all~\(1\leq p\leq n\) and~\(x,y\in\N\), we compute
 \[
   \Bigl(\sum_{k,l,r,s=1}^{n}E_{rs}\boxtimes b^{rs}_{kl}\boxtimes b^{kl}_{ij}\Bigr)
          (e_p\otimes\lambda'_{x}\otimes\lambda'_{y})
   = \sum_{k,l,r=1}^{n}\zeta^{d_{p}(d_{i}-d_{r}-d_{p}-d_{j})} 
      e_{r}\otimes b^{rp}_{kl}\lambda'_{x}\otimes b^{kl}_{ij}\lambda'_{y},
 \]
 and 
 \[
 \Bigl(\sum_{k,l,r,s=1}^{n}E_{rs}\boxtimes \Comult[B](b^{rs}_{ij})\Bigr)
  (e_{p}\otimes\lambda'_{x}\otimes\lambda'_{y})
 =\sum_{r=1}^{n} \zeta^{d_{p}(d_{i}-d_{r}-d_{p}-d_{j})} 
 e_{r} \otimes\Comult[B](b^{rp}_{ij})(\lambda'_{x}\otimes\lambda'_{y}).
 \]
Thus we obtain the restriction 
 of~\(\Comult[B]\) on~\(\{b^{ij}_{kl}\}_{i,j,k,l=1}^{n}\subseteq B\) given by 
 \begin{equation}
  \label{eq:comult-rest}
   \Comult[B](b^{ij}_{kl})=\sum_{r,s=1}^{n}b^{ij}_{rs}\boxtimes_{\zeta}b^{rs}_{kl}.
 \end{equation}

\begin{proof}[Proof of \textup{Theorem~\ref{the:univ_obj}}] 
Suppose~\(\G=(B,\beta,\Comult[B])\) is an object of~\(\mathcal{C}^{\T}(\Mat{n}(\C),\Pi,\phi)\). Denote the action of~\(\G\) on~\(\Mat{n}(\C)\) by~\(\eta\colon \Mat{n}(\C)\to\Mat{n}(\C)\boxtimes B\). Then~\(\eta\) satisfies the following condition:
 \begin{equation}
  \label{eq:trace-prsv}
  (\phi\boxtimes\Id_{B})\eta(M)=\phi(M)1_{B},
  \qquad\text{for all~\(M\in\Mat{n}(\C)\).}
 \end{equation}

 Suppose~\(\G[H]=\Bialg{C'}\) is the bosonisation of~\(\G=(B,\beta,\Comult[B])\) 
 and~\(\widetilde{\eta}\) be the corresponding action of~\(\G[H]\) on~\(\Cont(\T)\boxtimes_{\zeta}\Mat{n}(\C)\), in Proposition~\ref{prop:bos_act_lift} for~\((\Mat{n}(\C),\Pi)\). 
 On the elementary braided tensors~\(z^{r}\boxtimes_{\zeta}E_{ij}\) it is defined by
 \begin{equation}
  \label{ext_act-bos}
  \widetilde{\eta}(z^{r}\boxtimes_{\zeta}E_{ij}) 
  =\sum_{k,l=1}^{n}(z^{r}\boxtimes_{\zeta}E_{kl})\otimes (z^{d_{k}-d_{l}+r}\boxtimes_{\zeta}b^{kl}_{ij}),
  \quad\text{for all~\(r\in\Z, 1\leq i,j\leq n\),}
 \end{equation}
 where~\(\{b_{ij}^{kl}\}_{i,j,k,l=1}^{n}\) are the matrix elements of~\(\eta\) given by~\eqref{eq:Mn-action}. Also, Theorem~\ref{the:Boson} gives us the following commutation relation in~\(D=j_{1}(\Cont(\T))j_{2}(B)=\Cont(\T)\boxtimes_{\zeta}B\):
 \begin{equation}
  \label{eq:bos_gen_comm}
       (z\boxtimes_{\zeta}1_{B})(1_{\Cont(\T)}\boxtimes_{\zeta}b^{ij}_{kl})
       (z\boxtimes_{\zeta} 1_{B})^{*} 
   =  \zeta^{-(d_{k}-d_{i}+d_{j}-d_{l})}\cdot (1_{\Cont(\T)}\boxtimes_{\zeta} b^{ij}_{kl}).
 \end{equation}
for all~\(1\leq i,j,k,l\leq n\). The restriction of~\(\widetilde{\eta}\) on~\(\Mat{n}(\C)\) becomes 
\begin{equation}
 \label{eq:bos_act_res}
  \widetilde{\eta}(1_{\Cont(\T)}\boxtimes_{\zeta}E_{ij}) 
  =\sum_{k,l=1}^{n}(1_{\Cont(\T)}\boxtimes_{\zeta}E_{kl})\otimes (z^{d_{k}-d_{l}}\boxtimes_{\zeta}b^{kl}_{ij}),
  \quad\text{for all~\(1\leq i,j\leq n\).}
 \end{equation}
 Then we can verify that~\(\widetilde{\eta}|_{\Mat{n}(\C)}\) defines an action of~\(\G[H]\) on~\(\Mat{n}(\C)\) and preserves~\(\phi\). So, the matrix elements~\(\{z^{d_{i}-d_{j}}\boxtimes_{\zeta}b^{ij}_{lk}\}_{i,j,k,l=1}^{n}\) of~\(\widetilde{\eta}|_{\Mat{n}(\C)}\) must satisfy the relations~\eqref{eq:cond-2}-\eqref{eq:cond-6} with~\(\zeta=1\) by~\cite{W1998}*{Theorem 4.1}. Subsequently, using~\eqref{eq:bos_gen_comm} we can easily show that~\(\{b^{ij}_{kl}\}_{i,j,k,l=1}^{n}\) satisfy the relations~\eqref{eq:cond-2}-\eqref{eq:cond-6} in Theorem~\ref{the:qnt_aut_matrx}. Then the universal property of~\(A\) gives a unique \Star{}homomorphism~\(f\colon A\to B\) such that~\(f(u^{ij}_{kl})=b^{ij}_{kl}\) for all~\(i,j,k,l\in\{1,\cdots ,n\}\). Clearly, \(f\) is~\(\T\)\nb-equivariant and satisfies~\((f\boxtimes_{\zeta}f)\circ\Comult[A]=\Comult[B]\circ f\); hence~\(f\colon\Aut(\Mat{n}(\C),\Pi,\phi)\to\G\) is a braided compact quantum group homomorphism. 
\end{proof}
\subsection{Bosonisation as the inhomogeneous quantum symmetry group}
 \label{subsec:Bos_Drinf}
 \begin{definition}[\cite{BS2013}*{Definition 2.1}]
 Let \(D\) be a unital \(\Cst\)\nb-algebra and let~\(\tau_{D}\) be a faithful state on \(D\). An \emph{orthogonal filtration} for the pair \((D,\tau_{D})\) is a sequence of finite dimensional subspaces \(\{D_{i}\}_{i\in\mathcal{I}}\) of~\(D\), where~\(\mathcal{I}\) is the index set with a distinguished element~\(0\) such that \(D_{0}=\C\cdot 1_{D}\), \(\textup{Span}(\cup_{i\in \mathcal{I}}D_{i})\) is dense in \(D\) and \(\tau_{D}(a^{*}b)=0\) if \(a\in D_{i}\), \(b\in D_{j}\) and \(i\neq j\). We denote the triple~\((D,\tau_{D},\{D_{i}\}_{i\in\mathcal{I})})\) by~\(\widetilde{D}\).
\end{definition}
Let~\(\G[H]=\Bialg{Q}\) is a compact quantum group and let~\(\widetilde{D}\) be an orthogonal filtration of~\(D\). An action~\(\eta\colon D\to D\otimes Q\) is said to be \emph{filtration preserving} if~\(\eta(D_{i})\subseteq D_{i}\otimes_{\textup{alg}}Q\) for all~\(i\in\mathcal{I}\). Let~\(\mathcal{C}(\widetilde{D})\) be the category with objects as pairs~\((\G[H],\eta)\), where~\(\G[H]\) is a compact quantum group, \(\eta\) is a filtration preserving action of~\(\G[H]\) on~\(D\), and arrows are morphisms of compact quantum groups intertwining the respective actions. In~\cite{BS2013} it was also observed that if~\((\G[H],\varrho)\) is an object of~\(\mathcal{C}(\widetilde{D})\), then~\(\varrho\) preserves~\(\tau_{D}\).
\begin{theorem}[\cite{BS2013}*{Theorem 2.7}]
There exists a universal initial object in the category~\(\mathcal{C}(\widetilde{D})\), called the quantum symmetry group of the filtration~\(\widetilde{D}\) and denoted by~\(\textup{QISO}(\widetilde{D})\).
\end{theorem}
\begin{example}
 \label{ex:Mat_filt}
Define~\(\Mat{n}(\C)_{0}=\textup{Span}\{1_{\Mat{n}(\C)}\}\) and~\(\Mat{n}(\C)_{1}=\textup{Span}\{E_{ij}\mid 1\leq i,j\leq n\text{ and } i\neq j\}\). Then~\(\widetilde{\Mat{n}(\C)}=(\Mat{n}(\C),\phi,\{\Mat{n}(\C)_{i}\}_{i=0,1})\) is an orthogonal filtration. Let~\(\Qgrp{H}{Q}\) be a compact quantum group and let~\(\varrho\colon\Mat{n}(\C)\to\Mat{n}(\C)\otimes Q\) be an action of~\(\G\) on~\(\Mat{n}(\C)\). Then~\(\varrho\) is uniquely determined by the set elements~\(\{q^{kl}_{ij}\}_{1\leq i,j,k,l\leq n}\subseteq Q\) such that
\[
   \varrho(E_{ij})=\sum_{k,l=1}^{n}E_{kl}\otimes q^{kl}_{ij}, 
   \qquad
   \text{for all~\(1\leq i,j\leq n\).}
\]
If~\(\varrho\) preserves the orthogonal filtration~\(\widetilde{\Mat{n}(\C)}\) then~\(\varrho\) preserves the trace~\(\phi\). Conversely, if we assume that~\(\varrho\) preserves the canonical trace~\(\phi\), then~\(\sum_{i=1}^{n}q^{kl}_{ii}=\delta_{k,l}\) for all~\(1\leq k,l\leq n\). Thus~\(\varrho\) preserves the orthogonal filtration~\(\widetilde{\Mat{n}(\C)}\). Hence,~\(\textup{QISO}(\widetilde{\Mat{n}(\C)})\) coincides with Wang's \(\Aut(\Mat{n}(\C),\phi)\) in~\cite{W1998}*{Theorem 4.1}. 
\end{example}
\begin{example}
  \label{ex:Circle_filt}
  Let~\(z\) be the unitary generator of~\(\Cont(\T)\). Then the continuous linear extension of the map~\(z^{r}\to \delta_{r,0}\) for all~\(r\in\Z\) is the canonical trace~\(\tau_{\Cont(\T)}\) on~\(\Cont(\T)\cong\Cred(\Z)\) and the subspaces~\(\Cont(\T)_{i}\defeq\textup{Span}\{z^{i}\}\) for all~\(i\in\Z\) form an orthogonal filtration for~\((\Cont(\T),\tau_{\Cont(\T)})\). Let~\(\Qgrp{H}{Q}\) be a compact quantum group and let~\(\varrho\colon\Cont(\T)\to\Cont(\T)\otimes Q\) be an action of~\(\G\) on~\(\Cont(\T)\). Suppose, \((\G[H],\varrho)\)  
is an object of the category~\(\mathcal{C}(\widetilde{\Cont(\T)})\). Then~\(\varrho(z^{r})=z^{r}\otimes q_{r}\), where~\(q_{r}\in Q\) for all~\(r\in\Z\). Since, \(\varrho\) is a~\Star{}homomorphism then~\(q_{0}=1_{Q}\), \(q_{1}\) is unitary and~\(q_{r}=q_{1}^{r}\) for all~\(r\in\Z\). Moreover, \(\Comult[Q](q_{1})=q_{1}\otimes q_{1}\). Thus~\(\textup{QISO}(\widetilde{\Cont(\T)})\cong\Bialg{\Cont(\T)}\).
\end{example}
\begin{example}
 \label{ex:crossed_filt}
The braided tensor product~\(\Cont(\T)\boxtimes_{\zeta}\Mat{n}(\C)\) is isomorphic to the (reduced) crossed product~\(\Mat{n}(\C)\rtimes\Z\) for the~\(\Z\)\nb-action on~\(\Mat{n}(\C)\) generated by~\(\Pi_{\zeta^{-1}}\). Thus \(\Cont(\T)\boxtimes_{\zeta}\Mat{n}(\C)\) is the universal~\(\Cst\)\nb-algebra generated by a unitary~\(v\) and~\(\{E_{ij}\}_{i,j=1}^{n}\) subject to the relations~\eqref{eq:gen_Mat} and the commutation relation 
\begin{equation}
 \label{eq:cross-prod-univ}
  vE_{ij}v^{*}=\zeta^{d_{i}-d_{j}}E_{ij}.
\end{equation}
Composing~\(\phi\) with the conditional expectation~\(\Mat{n}(\C)\rtimes\Z\to\Mat{n}(\C)\) gives a state~\(\tau\) on~\(\Mat{n}(\C)\rtimes\Z\). Therefore,  
\(\tau\) is the continuous linear extension of the map \(v^{r}E_{ij}\to \frac{1}{n}\delta_{r,0}\cdot \delta_{i,j}\). By~\cite{BMRS2018}*{Proposition 4.9 \& Lemma 6.1} the triple~\(\widetilde{\Mat{n}(\C)\rtimes\Z}\defeq(\Mat{n}(\C)\rtimes\Z , \tau, \{v^{r}\cdot 1_{\Mat{n}(\C)}\}_{r\in\Z}\cup\{v^{s}\cdot E_{ij}\}_{s\in\Z, 1\leq i,j\leq n; i\neq j})\) is an orthogonal filtration.

Recall the bosonisation~\(\textup{Bos}(\Aut(\Mat{n}(\C),\Pi,\phi))=\Bialg{C}\) constructed in Theorem~\ref{the:Boson}. Then we apply Proposition~\ref{prop:bos_act_lift} to the canonical action~\(\eta\) of~\(\Aut(\Mat{n}(\C),\Pi,\phi)\) on~\((\Mat{n}(\C),\Pi)\) in Proposition~\ref{prop:action_qnt_on_Matx}. The resulting action~\(\widetilde{\eta}\colon \Mat{n}(\C)\rtimes\Z\to\Mat{n}(\C)\rtimes\Z\otimes C\) of~\(\Bialg{C}\) on~\(\Mat{n}(\C)\rtimes\Z\) is defined by
\[
  \widetilde{\eta}(v)=v\otimes z, 
  \qquad 
  \widetilde{\eta}(E_{ij})=\sum_{k,l=1}^{n} E_{kl}\otimes z^{d_{k}-d_{l}}u^{kl}_{ij}.
\]
Since~\(\widetilde{\eta}\) preserves the filtration \(\widetilde{\Mat{n}(\C)\rtimes\Z}\) the pair
\((\textup{Bos}(\Aut(\Mat{n}(\C),\Pi,\phi)),\widetilde{\eta})\) is an object of 
\(\mathcal{C}(\widetilde{\Mat{n}(\C)\rtimes\Z})\). 
\end{example}
\begin{proposition}
\label{prop:Boson_isometry}
 \(\textup{QISO}(\widetilde{\Mat{n}(\C)\rtimes\Z})\cong\textup{Bos}(\Aut(\Mat{n}(\C),\Pi,\phi))\).
\end{proposition}
\begin{proof}
Let~\(\Qgrp{H}{Q}\) be a compact quantum group. Suppose~\((\G[H],\varrho)\) is an 
object of~\(\mathcal{C}(\widetilde{\Mat{n}(\C)\rtimes\Z})\). Then~\(\varrho\) preserves the 
filtration described above. There exists a unitary 
\(u\in Q\) such that \(\varrho(v^{r})=v^{r}\otimes u^{r}\) and~\(\Comult[Q](u^{r})=u^{r}\otimes u^{r}\) 
for all~\(r\in\Z\). 
On the other hand, the restriction of~\(\varrho\) on~\(E_{ij}\) is given by 
\(\varrho(E_{ij})=\sum_{k,l=1}^{n}E_{kl}\otimes q^{kl}_{ij}\), where~\(q^{kl}_{ij}\in Q\) for all 
\(1\leq i,j,k,l\leq n\). Since~\(\varrho\) preserves~\(\tau\), the restriction 
\(\varrho|_{\Mat{n}(\C)}\) preserves~\(\phi\). By~\cite{W1998}*{Theorem 4.1}, the elements 
\(q^{ij}_{kl}\) staisfy~\eqref{eq:cond-2}-\eqref{eq:cond-6} with~\(\zeta=1\) and 
\[
  \Comult[Q](q_{ij}^{kl})=\sum_{r,s=1}^{n}q^{kl}_{rs}\otimes q^{rs}_{ij}, 
  \qquad
  \text{for all~\(1\leq i,j,k,l\leq n\).}
\]
Using the commutation relation~\eqref{eq:cross-prod-univ} and the condition that~\(\widetilde{\eta}\) is a 
\Star{}homomorphism we compute 
\[
  \zeta^{d_{i}-d_{j}}\varrho(E_{ij}) 
  =\varrho(vE_{ij}v^{*})
  =\sum_{k,l=1}^{n}vE_{kl}v^{*}\otimes uq^{kl}_{ij}u^{*}
  =\sum_{k,l=1}^{n}E_{kl}\otimes \zeta^{d_{k}-d_{l}}uq^{kl}_{ij}u^{*}.
\]
\sloppy
This implies~\(uq^{kl}_{ij}u^{*}=\zeta^{d_{k}-d_{i}+d_{j}-d_{l}}q_{ij}^{kl}\) for all~\(1\leq i,j,k,l\leq n\). Therefore, 
\(\textup{QISO}(\widetilde{\Mat{n}(\C)})\cong\Bialg{\mathcal{Q}}\), where~\(\mathcal{Q}\) is the universal 
\(\Cst\)\nb-algebra generated by the elements~\(u\), \(\{q^{kl}_{ij}\}_{i,j,k,l=1}^{n}\) and~\(\Comult[\mathcal{Q}]\) is 
the restriction of~\(\Comult[Q]\) to \({\mathcal{Q}}\). Then it is a routine check that 
\(\mathcal{Q}\) is also generated by elements~\(u\) and 
\(w^{kl}_{ij}\defeq u^{d_{l}-d_{k}}q^{kl}_{ij}\) for all~\(1\leq i,j,k,l\leq n\) such that~\(u\) is unitary, \(\{w^{kl}_{ij}\}\) satisfy 
the relations~\eqref{eq:cond-2}-\eqref{eq:cond-6} and~\(uw^{kl}_{ij}u^{*}=\zeta^{d_{k}-d_{i}+d_{j}-d_{l}}
w^{kl}_{ij}\). Moreover, \(\Comult[\mathcal{Q}](w^{kl}_{ij})=\sum_{r,s=1}^{n}w^{kl}_{rs}\otimes u^{d_{r}-d_{k}+d_{l}-d_{s}}w^{rs}_{ij}\) 
for all~\(1\leq i,j,k,l\leq n\). Thus~\(\Bialg{\mathcal{Q}}\) is the bosonisation~\(\Bialg{C}\) of~\(\Aut(\Mat{n}(\C),\Pi,\phi)\) in 
Theorem~\ref{the:Boson}.
\end{proof}
\sloppy
The action~\(\Pi(E_{ij})=E_{ij}\otimes z^{d_{i}-d_{j}}\) of~\(\T\) on~\(\Mat{n}(\C)\) preserves~\(\phi\), it also preserves~\(\widetilde{\Mat{n}(\C)}\) in Example~\ref{ex:Mat_filt}. Thus~\((\Bialg{\Cont(\T)},\Pi)\) is an object of~\(\mathcal{C}(\widetilde{\Mat{n}(\C)})\). Hence, there exists a unique compact quantum group homomorphism \(f\colon \textup{QISO}(\widetilde{\Mat{n}(\C)})\to\Cont(\T)\) such that~\(f(q_{ij}^{kl})=\delta_{k,i}\cdot \delta_{j,l}\cdot z^{d_{i}-d_{j}}\), where~\(\{q^{ij}_{kl}\}_{1\leq i,j,k,l\leq n}\) is a generating set of Wang's \(\Aut(\Mat{n}(\C),\phi)\).

Now we view the \(\Rmattxt\)\nb-matrix in~\eqref{eq:Rmat} as a unitary element of the multiplier algebra of~\(\Contvin(\Z)\otimes\Contvin(\Z)\). Then~\(\bichar\defeq (\Id_{\Contvin(\Z)}\otimes \hat{f})\Rmat\) is a bicharacter, where~\(\hat{f}\) is the dual quantum group homomorphism from \(\Contvin(\Z)\) to the dual of~\(\widetilde{\textup{QISO}(\Mat{n}(\C))}\) (see~\cite{MRW2012}*{Section 3 \& 4}). 
By~\cite{BMRS2018}*{Theorem 5.1}, 
\(\textup{QISO}(\widetilde{\Mat{n}(\C)\rtimes\Z})\cong\mathfrak{D}_{\bichar}\), 
where~\(\mathfrak{D}_{\bichar}\) is the generalised Drinfeld's double of~\(\textup{QISO}(\widetilde{\Mat{n}(\C)})\) and~\(\textup{QISO}(\widetilde{\Cont(\T)})\) (see~\cite{R2015}). Combining it with Proposition~\ref{prop:Boson_isometry} we obtain the following corollary.
\begin{corollary}
 \label{cor:bos_dinf}
  \(\mathfrak{D}_{\bichar}\cong \textup{Bos}(\Aut(\Mat{n}(\C),\Pi,\phi))\).
\end{corollary}

\begin{remark}
 \label{rem:inhomogeneous}
 For~\(i,j,k,l\in \{1,\cdots , n\}\), the elements  
  \(w_{ij}^{kl}\) generate the copy of~\(\Aut(\Mat{n}(\C),\Pi,\phi)\) inside~\(\textup{Bos}(\Aut(\Mat{n}(\C),\Pi,\phi))\), whereas the elements~\(q_{ij}^{kl}=u^{d_{k}-d_{l}}w^{kl}_{ij}\) generate the copy of \(\Aut(\Mat{n}(\C),\phi)\) inside~\(\mathfrak{D}_{\bichar}\). 
  
  Then~\(q_{ij}^{kl}=u^{d_{k}-d_{l}}w_{ij}^{kl}\in C\) 
  is a homogeneous element of degree~\((d_{k}-d_{l})+(d_{k}-d_{i}+d_{j}-d_{l})\) for the diagonal action~\(\Comult[\Cont(\T)]\bowtie\alpha\). Similarly, an element~\(v^{r}E_{ij}\in\Mat{n}(\C)\rtimes\Z\cong\Cont(\T)\boxtimes_{\zeta}
  \Mat{n}(\C)\) is homogeneous of degree~\(r+d_{i}-d_{j}\) for the 
  diagonal action~\(\Comult[\Cont(\T)]\bowtie\Pi\). 
  
  The restriction of the action~\(\varrho\) on the factors~\(\Cont(\T)\) and 
  \(\Mat{n}(\C)\) in the above Proposition~\ref{prop:Boson_isometry} is precisely 
  the action of~\(\mathfrak{D}_{\bichar}=\Bialg{\mathcal{Q}}\) on those factors 
  (see~\cite{BMRS2018}*{Theorem 3.5}). The latter one is given by
  \[
    \varrho|_{\Mat{n}(\C)}(E_{ij}) 
    =\sum_{k,l=1}^{n}E_{kl}\otimes q^{kl}_{ij}
    =\sum_{k,l=1}^{n}E_{kl}\otimes u^{d_{k}-d_{l}}w^{kl}_{ij}.
  \]
  However, \(\varrho|_{\Mat{n}(\C)}\) is not~\(\T\)\nb-equivariant for the actions~\(\Pi\) and 
  \(\Comult[\Cont(\T)]\bowtie\alpha\). This happens because~\(\Aut(\Mat{n}(\C),\phi)\) fails to capture \(\T\)\nb-homomgeneous compact quantum symmetries of the dynamical system~\((\Mat{n}(\C),\Z,\Pi_{\zeta^{-1}})\).

  On the other hand, if we realise~\(\varrho\) as an action of~\(\textup{Bos}(\Aut(\Mat{n}(\C),\Pi,\phi)\) on~\(\Mat{n}(\C)\rtimes\Z\), then it coincides with~\(\eta\) in Proposition~\ref{prop:bos_act_lift}. Consequently,  
  \(\varrho |_{\Mat{n}(\C)}\) gets identified with~\(\eta\) in Proposition~\ref{prop:action_qnt_on_Matx}. In this sense, \(\Aut(\Mat{n}(\C),\Pi,\phi)\) is the \(\T\)\nb-homogeneous quantum symmetry group of the \(\Cst\)\nb-dynamical system~\((\Mat{n}(\C),\Z,\Pi_{\zeta^{-1}})\) and \(\textup{Bos}(\Aut(\Mat{n}(\C),\Pi,\phi))\) is the inhomogeneous quantum symmetry group of the same system.
 \end{remark}
 \section{Towards the direct sum of matrix algebras}
  \label{sec:direct_sum}
  For a fixed~\(m\in\N\), let~\(n_{1},n_{2},\cdots , n_{m}\in\N\). Then~\(D=\oplus_{x=1}^{m}\Mat{n_{x}}(\C)\) is the universal \(\Cst\)\nb-algebra generated by~\(\{E_{ij,x}\}_{1\leq i,j\leq n_{x}, 1\leq x\leq m}\) subject to the following conditions:
  \begin{equation}
   \label{eq:mat-dir-sum}
   E_{ij,x}E_{kl,y}=\delta_{j,k}\cdot\delta_{x,y}E_{il,x}, 
   \qquad 
   E_{ij,x}^{*}=E_{ji,x},
   \qquad 
   \sum_{x=1}^{m}\sum_{i=1}^{n_{x}}E_{ii,x}=1 .
  \end{equation}
  For each~\(x\in\{1,\cdots ,m\}\), we fix \(d^{x}_{1}\leq d^{x}_{2}\leq \cdots \leq d^{x}_{n_{x}}\in\Z\) and define an action~\(\Pi^{x}\) of~\(\T\) 
  on the~\(x^\text{th}\) component of~\(D\) by~\(\Pi^{x}_{z}(E_{ij,p})\defeq z^{d_{i}^{x}-d_{j}^{x}}E_{ij,x}\) for~\(z\in\T\). Extending them, we obtain the following \(\T\)\nb-action on~\(D\): \(\overline{\Pi}_{z}(\oplus_{x=1}^{m} M_{p})=\oplus_{x=1}^{m}\Pi^{x}_{z}(M_{x})\) for all~\(M_{x}\in\Mat{n_{x}}(\C)\), and \(z\in\T\). The map~\(\phi(E_{ij,x})\defeq \delta_{i,j}\) is a positive linear functional on~\(D\) and~\(\overline{\Pi}\) preserves~\(\phi\).
  \begin{theorem}
   \label{the:Aut_dirsum}
   Consider the monoidal category~\((\Cstcat(\T),\boxtimes_{\zeta})\) for a 
   fixed~\(\zeta\in\T\). Let~\(A\) be the universal \(\Cst\)\nb-algebra with generators 
   \(u^{ij}_{kl,xy}\) for \(1\leq i,j\leq n_{x},1\leq k,l\leq n_{y}, 1\leq x,y \leq m\) 
   and the following relations:
   \begin{align}
   \nonumber
 &\sum_{t=1}^{n_{x}}(\zeta^{d_{t}^{x}(d^{y}_{k}-d^{x}_{i}+d^{x}_{t}-d^{y}_{s})}u^{it}_{ks,xy})\cdot 
 (\zeta^{d^{x}_{j}(d^{w}_{m}-d^{x}_{t}+d^{x}_{j}-d^{w}_{l})}u^{tj}_{ml,xw})\\ 
  \label{eq:dirsum_cond1}
 &=\delta_{y,w}\cdot \delta_{s,m}\cdot (\zeta^{d^{x}_{j}(d^{y}_{k}-d^{x}_{i}+d^{x}_{j}-d^{y}_{l}}u^{ij}_{kl,xy}),
 \\ \nonumber
 &\text{for all \(1\leq k,s\leq n_{y}\), \(1\leq m,l\leq n_{w}\), \(1\leq i,j\leq n_{x}\), 
 \(1\leq x,y,w\leq m\)\textup{;}} \\
 \nonumber
 & \sum_{t=1}^{n_{x}} (\zeta^{d^{y}_{s}(d^{x}_{k}-d^{y}_{i}+d^{y}_{s}-d^{x}_{t})}u^{is}_{kt,yx})\cdot 
 (\zeta^{d^{w}_{j}(d^{x}_{t}-d^{w}_{m}+d^{w}_{j}-d^{x}_{l})}u^{mj}_{tl,wx})\\ 
  \label{eq:dirsum_cond2}
 &=\delta_{y,w}\cdot \delta_{s,m}\cdot (\zeta^{d^{y}_{j}(d^{x}_{k}-d^{y}_{i}+d^{y}_{j}-d^{x}_{l})}u^{ij}_{kl,yx}),
 \\ \nonumber
& \text{for all \(1\leq i,s \leq n_{y}\), \(1\leq m,j\leq n_{w}\), \(1\leq k,l\leq n_{x}\), 
\(1\leq x,y,w\leq m\)\textup{;}} \\
 \label{eq:dirsum_cond3}
  & u^{ij}_{kl,xy}{ }^{*}
   = \zeta^{(d^{x}_{i}-d^{x}_{j})(d^{y}_{l}-d^{x}_{j}+d^{x}_{i}-d^{y}_{k})}(u^{ji}_{lk,xy}),
   \\ \nonumber 
  & \text{for all~\(1\leq i,j \leq n_{x}\), \(1\leq k,l\leq n_{y}\), \(1\leq x,y\leq m\)\textup{;}} \\
  \label{eq:dirsum_cond4}
  & \sum_{y=1}^{m}\sum_{r=1}^{n_{y}}u^{ij}_{rr,xy}
  =\delta_{i,j},
  \quad \text{for all~\(1\leq i,j\leq n_{x}\), \(1\leq x\leq m\) \textup{;}} \\
  \label{eq:dirsum_cond5}
  & \sum_{x=1}^{m}\sum_{r=1}^{n_{x}}u^{rr}_{kl,xy}
  =\delta_{k,l},
  \quad\text{for all~\(1\leq k,l\leq n_{y}\), \(1\leq y\leq m\) \textup{.}}
\end{align}
 There exists a unique continuous action~\(\alpha\) of~\(\T\) on~\(A\) satisfying
   \begin{equation}
     \label{eq:dirsum_cond6}
     \alpha_{z}(u^{ij}_{kl,xy})=z^{d_{k}^{y}-d_{i}^{x}+d_{j}^{x}-d_{l}^{y}} u^{ij}_{kl,xy},
     \qquad\text{for all~\(z\in\T\).}
   \end{equation}	
   Thus~\((A,\alpha)\) is an object of~\(\Cstcat(\T)\). Moreover, there exists a unique 
   \(\Comult[A]\in\Mor^{\T}(A,A\boxtimes_{\zeta}A)\) such that~\((A,\alpha,\Comult[A])\) is a braided compact quantum group over~\(\T\). We denote it by~\(\Aut(D,\overline{\Pi},\phi)\). 
 The map~\(\eta\colon D\to D\boxtimes_{\zeta}A\), defined by
 \[
   \eta(E_{ij,x})=\sum_{y=1}^{m}\sum_{r,s=1}^{n_{y}}E_{rs,y}\boxtimes_{\zeta}u^{rs}_{ij,yx} 
\]
 extends to a faithful \(\phi\)\nb-preserving action of~\(\Aut(D,\overline{\Pi},\phi)\) on~\((D,\overline{\Pi})\).
  \end{theorem}
  \begin{proof}
  We will essentially use the same methods as in Section~\ref{sec:act_brd_cpt} 
  to prove this result. Let~\(\mathcal{A}\) be the \Star{}algebra generated by the 
  elements~\(u^{ij}_{kl,xy}\) satisfying~\eqref{eq:dirsum_cond1}-\eqref{eq:dirsum_cond5}.
  We identify~\(\textup{End}(D)\) with \(V\defeq\oplus_{x,y=1}^{m}\textup{End}(\Mat{n_{x}\times n_{y}}(\C))\) via the vector space isomorphism 
  \(E_{ij,x}\otimes E_{kl,y}\to E_{ik}\otimes E_{jl}\otimes E_{xy}\). The \(\T\)\nb-action~\(\overline{\Pi}\) 
  on~\(D\) induces a~\(\T\) action on~\(V\) defined by 
  \(E_{ik}\otimes E_{jl}\otimes E_{xy}\mapsto 
  z^{d^{x}_{i}-d^{y}_{k}-d^{x}_{j}+d_{l}^{y}} E_{ik}\otimes E_{jl}\otimes E_{xy}\) for all~\(z\in\T\).
 Define~\(u\in V\otimes\mathcal{A}\) by 
  \[
    u=\sum_{x,y=1}^{m}\sum_{k,l=1}^{n_{y}}\sum_{i,j=1}^{n_{x}} 
    E_{ik}\otimes E_{jl}\otimes E_{xy}\otimes u^{ij}_{kl,xy}.
  \]
 Using the relations~\eqref{eq:dirsum_cond1}-\eqref{eq:dirsum_cond5}, 
 we can show that~\(u\) is unitary, and~\(A\) is the completion of~\(\mathcal{A}\) 
 with respect to the largest \(\Cst\)\nb-seminorm on~\(\mathcal{A}\).  
 Similarly, we can show that~\(\alpha\) in~\eqref{eq:dirsum_cond6} defines a continuous action of~\(\T\) on~\(A\), the following formulae define~\(\Comult[A]\):
  \begin{equation}
   \label{eq:brd_comult_dirsum}
   \Comult[A](u^{ij}_{kl,xy})
    =\sum_{w=1}^{m}\sum_{r,s=1}^{n_{w}}u^{ij}_{rs,xw}\boxtimes_{\zeta} u^{rs}_{kl,wy},
 \qquad\text{\(1\leq i,j,k,l\leq n\),}
\end{equation}
such that \(\Aut(D,\overline{\Pi},\phi)=(A,\alpha,\Comult[A])\) is a 
braided compact quantum group over~\(\T\), and  \(\eta\) defines a faithful \(\phi\)\nb-preserving action of~\(\Aut(D,\overline{\Pi},\phi)\) on~\((D,\overline{\Pi})\).
 \end{proof}  
 \begin{remark}
  \label{rem:comm_finpt}
   In particular, if~\(m=1\) then the Theorem~\ref{the:Aut_dirsum} above reduces to 
 Theorem~\ref{the:qnt_aut_matrx}. On the other hand, 
 if~\(n_{x}=1\) for~\(1\leq x\leq m\), then~\(D\cong\C^{m}\) and 
 the action~\(\overline{\Pi}\) is trivial. 
 Define~\(a_{xy}=u^{11}_{11,xy}\) for all~\(1\leq x,y\leq m\). 
 Indeed, \(A\) is generated by the entires of the matrix~\(u=(a_{xy})\) 
 and the relations~\eqref{eq:dirsum_cond1}-\eqref{eq:dirsum_cond5} show 
 that~\(\Aut(D,\overline{\Pi},\phi)\) becomes isomorphic to the 
 Wang's quantum permutation group of~\(m\) points.
\end{remark}
 Following Example~\ref{ex:Mat_filt}, we observe that 
 \(\widetilde{D}=(D,\phi,\{D_{0}\}\cup\{D_{1,x}\}_{x=1}^{m})\) is 
 an orthogonal filtration for~\((D,\phi)\), where~\(D_{0}=\textup{Span}\{1_{D}\}\) 
 and~\(D_{1,x}=\textup{Span}\{E_{ij,x}\mid 1\leq i,j\leq n_{x}\text{ and } i\neq j\}\). 
 Combining it with Example~\eqref{ex:Circle_filt} and 
  we construct an orthogonal 
 filtration~\(\widetilde{D\rtimes\Z}\) on the crossed product~\(D\rtimes\Z\), for 
 the action of~\(\Z\) on~\(D\) generated by~\(\overline{\Pi}_{\zeta^{-1}}\), as in 
 Example~\ref{ex:Circle_filt}. Then by the 
 arguments analogous to those of Section~\ref{sec:Boson} and Section~\ref{sec:univ-prop}
we have the following theorem.
 \begin{theorem}
  \label{the:dirsum_bos_univ}
  Consider the categories~\(\mathcal{C}^{\T}(D,\overline{\Pi},\phi)\) and 
  \(\mathcal{C}(\widetilde{D\rtimes\Z})\). Then the following statements hold true.
  \begin{enumerate}
   \item \(\Aut(D,\overline{\Pi},\phi)\) is the universal initial object of the category 
   \(\mathcal{C}^{\T}(D,\overline{\Pi},\phi)\);
   \item Suppose~\(\Bialg{C}\) is the bosonisation of 
   \(\Aut(D,\overline{\Pi},\phi)\). Then~\(C\) is the universal~\(\Cst\)\nb-algebra 
   generated by the elements~\(z\) and~\(u^{ij}_{kl,xy}\) for 
   \(1\leq i,j\leq n_{x},1\leq k,l\leq n_{y}, 1\leq x,y \leq m\) subject to the 
   relations~\(z^{*}z=zz^{*}=1\), \eqref{eq:dirsum_cond1}-\eqref{eq:dirsum_cond5}, 
   and~\(zu^{ij}_{kl,xy}z^{*}=\zeta^{d^{x}_{i}-d^{y}_{k}-d^{y}_{l}+d^{x}_{j}} u^{ij}_{kl,xy}\). The comultiplication map~\(\Comult[C]\colon C\to C\otimes C\) is given by 
   \[
    \Comult[C](z)=z\otimes z, 
    \qquad 
    \Comult[C](u^{ij}_{kl,xy})
    =\sum_{w=1}^{m}\sum_{r,s=1}^{n_{w}}u^{ij}_{rs,xw}\otimes 
    z^{d^{w}_{r}-d^{x}_{i}+d^{x}_{j}-d^{w}_{s}}u^{rs}_{kl,wy}.
   \]
   We denote~\(\Bialg{C}\) by~\(\textup{Bos}(\Aut(D,\overline{\Pi},\phi))\).
   \item \(\Aut(D,\overline{\Pi},\phi)\rtimes\Z\) is the universal, initial object of the category 
   \(\mathcal{C}(\widetilde{D\rtimes\Z})\), that is, \(\textup{QISO}(\widetilde{D\rtimes\Z})\cong \textup{Bos}(\Aut(D,\overline{\Pi},\phi))\).
  \end{enumerate}
 \end{theorem}


\begin{bibdiv}
  \begin{biblist}
  \bib{B2005}{article}{
  author={Banica, Teodor},
  title={Quantum automorphism groups of homogeneous graphs},
  date={2005},
  issn={0022-1236},
  journal={J. Funct. Anal.},
  volume={224},
  number={2},
  pages={243\ndash 280},
  doi={10.1016/j.jfa.2004.11.002},
}

\bib{B2005a}{article}{
  author={Banica, Teodor},
  title={Quantum automorphism groups of small metric spaces},
  date={2005},
  issn={0030-8730},
  journal={Pacific J. Math.},
  volume={219},
  number={1},
  pages={27\ndash 51},
  doi={10.2140/pjm.2005.219.27},
}

\bib{BG2010a}{article}{
  author={Banica, Teodor},
  author={Goswami, Debashish},
  title={Quantum isometries and noncommutative spheres},
  date={2010},
  issn={0010-3616},
  journal={Comm. Math. Phys.},
  volume={298},
  number={2},
  pages={343\ndash 356},
  doi={10.1007/s00220-010-1060-5},
}

\bib{BS2013}{article}{
  author={Banica, Teodor},
  author={Skalski, Adam},
  title={Quantum symmetry groups of {$C^\ast $}-algebras equipped with orthogonal filtrations},
  date={2013},
  issn={0024-6115},
  journal={Proc. Lond. Math. Soc. (3)},
  volume={106},
  number={5},
  pages={980\ndash 1004},
  doi={10.1112/plms/pds071},
}

\bib{BG2009a}{article}{
  author={Bhowmick, Jyotishman},
  author={Goswami, Debashish},
  title={Quantum group of orientation-preserving {R}iemannian isometries},
  date={2009},
  issn={0022-1236},
  journal={J. Funct. Anal.},
  volume={257},
  number={8},
  pages={2530\ndash 2572},
  doi={10.1016/j.jfa.2009.07.006},
}

\bib{BMRS2018}{article}{
  author={Bhowmick, Jyotishman},
  author={Mandal, Arnab},
  author={Roy, Sutanu},
  author={Skalski, Adam},
  title={Quantum Symmetries of the Twisted Tensor Products of C*-Algebras},
   journal={Comm. Math. Phys.},
   volume={368},
   date={2019},
   number={3},
   pages={1051--1085},
   issn={0010-3616},
   doi={10.1007/s00220-018-3279-5},
}

\bib{B2003}{article}{
  author={Bichon, Julien},
  title={Quantum automorphism groups of finite graphs},
  date={2003},
  issn={0002-9939},
  journal={Proc. Amer. Math. Soc.},
  volume={131},
  number={3},
  pages={665\ndash 673},
  doi={10.1090/S0002-9939-02-06798-9},
}

\bib{G2009}{article}{
  author={Goswami, Debashish},
  title={Quantum group of isometries in classical and noncommutative geometry},
  date={2009},
  issn={0010-3616},
  journal={Comm. Math. Phys.},
  volume={285},
  number={1},
  pages={141\ndash 160},
  doi={10.1007/s00220-008-0461-1},
}

\bib{G2020}{article}{
  author={Goswami, Debashish},
  title={Non-existence of genuine (compact) quantum symmetries of compact, connected smooth manifolds},
  date={2020},
  issn={0001-8708},
  journal={Adv. Math.},
  volume={369},
  pages={107181, 25},
  doi={10.1016/j.aim.2020.107181},
}

\bib{GJ2018}{article}{
  author={Goswami, Debashish},
  author={Joardar, Soumalya},
  title={Non-existence of faithful isometric action of compact quantum groups on compact, connected {R}iemannian manifolds},
  date={2018},
  issn={1016-443X},
  journal={Geom. Funct. Anal.},
  volume={28},
  number={1},
  pages={146\ndash 178},
  doi={10.1007/s00039-018-0437-z},
}

\bib{JM2018}{article}{
  author={Joardar, Soumalya},
  author={Mandal, Arnab},
  title={Quantum symmetry of graph {$C^*$}-algebras associated with connected graphs},
  date={2018},
  issn={0219-0257},
  journal={Infin. Dimens. Anal. Quantum Probab. Relat. Top.},
  volume={21},
  number={3},
  pages={1850019, 18},
  doi={10.1142/S0219025718500194},
}

\bib{KMRW2016}{article}{
  author={Kasprzak, Pawe\l },
  author={Meyer, Ralf},
  author={Roy, Sutanu},
  author={Woronowicz, Stanis\l aw~Lech},
  title={Braided quantum {$\rm SU(2)$} groups},
  date={2016},
  issn={1661-6952},
  journal={J. Noncommut. Geom.},
  volume={10},
  number={4},
  pages={1611\ndash 1625},
  doi={10.4171/JNCG/268},
}

\bib{M1994}{article}{
  author={Majid, Shahn},
  title={Cross products by braided groups and bosonization},
  date={1994},
  issn={0021-8693},
  journal={J. Algebra},
  volume={163},
  number={1},
  pages={165\ndash 190},
  doi={10.1006/jabr.1994.1011},
}

\bib{MR2019a}{article}{
  author={Meyer, Ralf},
  author={Roy, Sutanu},
  title={Braided free orthogonal quantum groups},
  date={2021},
  jounral={Int. Math. Res. Not. IMRN},
  doi={https://doi.org/10.1093/imrn/rnaa379},
}


\bib{MRW2012}{article}{
  author={Meyer, Ralf},
  author={Roy, Sutanu},
  author={Woronowicz, Stanis\l aw~Lech},
  title={Homomorphisms of quantum groups},
  date={2012},
  issn={1867-5778},
  journal={M\"{u}nster J. Math.},
  volume={5},
  pages={1\ndash 24},
  eprint={http://nbn-resolving.de/urn:nbn:de:hbz:6-88399662599},
}

\bib{MRW2014}{article}{
  author={Meyer, Ralf},
  author={Roy, Sutanu},
  author={Woronowicz, Stanis\l aw~Lech},
  title={Quantum group-twisted tensor products of {C{$^*$}}-algebras},
  date={2014},
  issn={0129-167X},
  journal={Internat. J. Math.},
  volume={25},
  number={2},
  pages={1450019, 37},
  doi={10.1142/S0129167X14500190},
}

\bib{MRW2016}{article}{
  author={Meyer, Ralf},
  author={Roy, Sutanu},
  author={Woronowicz, Stanis\l aw~Lech},
  title={Quantum group-twisted tensor products of {${\rm C}^*$}-algebras. {II}},
  date={2016},
  issn={1661-6952},
  journal={J. Noncommut. Geom.},
  volume={10},
  number={3},
  pages={859\ndash 888},
  doi={10.4171/JNCG/250},
}

\bib{R1985}{article}{
  author={Radford, David~E.},
  title={The structure of {H}opf algebras with a projection},
  date={1985},
  issn={0021-8693},
  journal={J. Algebra},
  volume={92},
  number={2},
  pages={322\ndash 347},
  doi={10.1016/0021-8693(85)90124-3},
}

\bib{RR2021}{article}{
  author={Rahaman, Atibur},
  author={Roy, Sutanu},
  title={Quantum $E(2)$ groups for complex deformation parameters},
  date={2021},
  journal={Rev. Math. Phys.},
  volume={33},
  pages={2250021, 28},
  doi={https://doi.org/10.1142/S0129055X21500215},
}

\bib{R2015}{article}{
      author={Roy, Sutanu},
       title={The {D}rinfeld double for {$C^*$}-algebraic quantum groups},
        date={2015},
        ISSN={0379-4024},
     journal={J. Operator Theory},
      volume={74},
      number={2},
       pages={485\ndash 515},
         doi={10.7900/jot.2014sep04.2053},
      review={\MR{3431941}},
}

\bib{SW2018}{article}{
  author={Schmidt, Simon},
  author={Weber, Moritz},
  title={Quantum symmetries of graph {$C^*$}-algebras},
  date={2018},
  issn={0008-4395},
  journal={Canad. Math. Bull.},
  volume={61},
  number={4},
  pages={848\ndash 864},
  doi={10.4153/CMB-2017-075-4},
}

\bib{W1998}{article}{
  author={Wang, Shuzhou},
  title={Quantum symmetry groups of finite spaces},
  date={1998},
  issn={0010-3616},
  journal={Comm. Math. Phys.},
  volume={195},
  number={1},
  pages={195\ndash 211},
  doi={10.1007/s002200050385},
}

\bib{W1998a}{incollection}{
  author={Woronowicz, Stanis\l aw~Lech},
  title={Compact quantum groups},
  date={1998},
  booktitle={Sym\'{e}tries quantiques ({L}es {H}ouches, 1995)},
  publisher={North-Holland, Amsterdam},
  pages={845\ndash 884},
}

\bib{W1987a}{article}{
  author={Woronowicz, Stanis\l aw~Lech},
  title={Compact matrix pseudogroups},
  date={1987},
  issn={0010-3616},
  journal={Comm. Math. Phys.},
  volume={111},
  number={4},
  pages={613\ndash 665},
}
  \end{biblist}
\end{bibdiv}
\end{document}